\documentclass[11pt]{amsart}

\usepackage[draft]{changes}

\definechangesauthor[color=red]{pe}
\definechangesauthor[color=red]{xm}
\definechangesauthor[color=red]{hx}
\newcommand{\stkout}[1]{\ifmmode\text{\sout{\ensuremath{#1}}}\else\sout{#1}\fi}
\setdeletedmarkup{\stkout{#1}}

\usepackage{amssymb}

\usepackage{graphics}
\usepackage{graphicx}

\usepackage{amsmath}
\usepackage{amsthm}
\usepackage{amsfonts}
\usepackage{mathtools}
\usepackage{xcolor}
\usepackage[english]{babel}
\usepackage[margin=1.0in]{geometry}
\usepackage[colorlinks=true,
        linkcolor=blue]{hyperref}
\usepackage{bbm}
\usepackage{verbatim}
\usepackage{extarrows}

\usepackage[T1]{fontenc}

\numberwithin{equation}{section}

\newtheorem{prop}{Proposition}
\newtheorem{lemma}[prop]{Lemma}

\newtheorem{thm}[prop]{Theorem}
\newtheorem{cor}[prop]{Corollary}

\numberwithin{prop}{section}

\newtheorem{defn}[prop]{Definition}
\theoremstyle{definition}

\newtheorem{rmk}[prop]{Remark}

\definecolor{c1}{rgb}{0.2,0.4,0.5}
\definecolor{c2}{rgb}{0.1,0.3,0.5}
\definecolor{c3}{rgb}{0.2,0.7,0.5}
\usepackage{tikz}

\def \k {K\"ahler }
\newcommand{\oo}[1]{\overline{#1}}

\newcommand{\bC}{\mathbb{C}}
\newcommand{\bR}{\mathbb{R}}
\newcommand{\bB}{\mathbb{B}}
\newcommand{\bK}{\mathbb{K}}

\newcommand{\eps}{\varepsilon}

\newcommand{\id}{\mathrm{id}}

\DeclareMathOperator{\Aut}{Aut}

\DeclareMathOperator{\Reg}{Reg}

\begin{document}

\title[]{Algebraicity of the Bergman Kernel}

\begin{abstract}
Our main result introduces a new way to characterize two-dimensional finite ball quotients by algebraicity of their Bergman kernels. This characterization is particular to dimension two and fails in higher dimensions, as is illustrated by a counterexample in dimension three constructed in this paper. As a corollary of our main theorem, we prove, e.g., that a smoothly bounded strictly pseudoconvex domain $G$ in $\mathbb{C}^2$ has rational Bergman kernel if and only if there is a rational biholomorphism from $G$ to $\mathbb{B}^2$.
\end{abstract}



\author [Ebenfelt]{Peter Ebenfelt}
\address{Department of Mathematics, University of California at San Diego, La Jolla, CA 92093, USA} \email{{pebenfelt@ucsd.edu}}

\author[Xiao]{Ming Xiao}
\address{Department of Mathematics, University of California at San Diego, La Jolla, CA 92093, USA}
\email{{m3xiao@ucsd.edu}}

\author [Xu]{Hang Xu}
\address{Department of Mathematics, University of California at San Diego, La Jolla, CA 92093, USA}
\email{{h9xu@ucsd.edu}}

\thanks{The first and second authors were supported in part by the NSF grants DMS-1900955 and DMS-1800549, respectively.}

\maketitle

\section{Introduction}

The Bergman kernel, introduced by S. Bergman in \cite{Bergman1933, Bergman1935} for domains in $\bC^n$ and later cast in differential geometric terms by S. Kobayashi \cite{Ko}, plays a fundamental role in several complex variables and complex geometry. Its biholomorphic invariance properties and intimate connection with the CR geometry of the boundary make it an important tool in the study of open complex manifolds. The use of the Bergman kernel, e.g., in the study of biholomorphic mappings and the geometry of bounded strictly pseudoconvex domains in $\bC^n$ was pioneered by C. Fefferman \cite{Fe,Fe2, Fe3}, who developed a theory of Bergman kernels in such domains and initiated a now famous program to describe the boundary singularity in terms of the local invariant CR geometry; see also \cite{BaileyEastwoodGraham1994}, \cite{Hirachi2000} for further progress on Fefferman's program.

A broad and general problem of foundational importance is that of classifying complex manifolds, or more generally analytic spaces, in terms of their Bergman kernels or Bergman metrics.  For example, a well-known result of Q. Lu \cite{Lu} implies that if a relatively compact domain in an $n$-dimensional K\"ahler manifold has a complete Bergman metric with constant holomorphic sectional curvature, then the domain is biholomorphic to the unit ball $\bB^n$ in $\bC^n$. Another example is the conjecture of S.-Y. Cheng \cite{CH}, which states that the Bergman metric of a smoothly bounded strongly pseudoconvex domain in $\bC^n$  is K\"ahler--Einstein (i.e., has Ricci curvature equal to a constant multiple of the metric tensor) if and only if it is biholomorphic to the unit ball $\bB^n$. This conjecture was confirmed by Fu-Wong \cite{FuWo} and Nemirovski--Shafikov \cite{NeSh} in the two dimensional case, and in the higher dimensional case by X. Huang and the second author \cite{HX}.

In this paper,
we introduce a new characterization of the two-dimensional unit ball $\bB^2\subset \bC^2$ and, more generally, two-dimensional finite ball quotients $\mathbb{B}^2/\Gamma$ in terms of algebraicity of the Bergman kernel. It is interesting, and perhaps surprising then, to note that such a characterization fails in the higher dimensional case. Indeed, in Section \ref{Sec counterexample} below we construct a relatively compact domain $G$ with smooth strongly pseudoconvex boundary in a three-dimensional algebraic variety $V\subset \mathbb{C}^4$, with an isolated normal singularity in the interior of $G$, such that the boundary $\partial G$ is not spherical and, furthermore, $G$ is not biholomorphic to any finite ball quotient; recall that a CR hypersurface $M$ of dimension $2n-1$ is said to be {\em spherical} if near each point $p\in M$, it is locally CR diffeomorphic to an open piece of the unit sphere $S^{2n-1}\subset \bC^n$. Nevertheless, in two dimensions it turns out that algebraicity of the Bergman kernel does characterize finite ball quotients:

\begin{thm}\label{main theorem intro}
	Let $V$ be a $2$-dimensional algebraic variety in $\mathbb{C}^N$, and $G$ a relatively compact domain in $V$. Assume that every point in $\overline{G}$ is a smooth point of $V$ except for finitely many isolated normal singularities inside $G$, and that $G$ has a smooth strongly pseudoconvex boundary. Then the Bergman kernel form of $G$ is algebraic if and only if there is an algebraic branched covering map $F$ from $\mathbb{B}^2$ onto $G$, which realizes $G$ as a ball quotient $\mathbb{B}^2/\Gamma$ where $\Gamma$ is a finite unitary group with no fixed points on $\partial \mathbb{B}^2$.
\end{thm}

\begin{rmk}\label{rmk counterexample intro} We note that in addition to showing that Theorem \ref{main theorem intro} fails in dimension $\geq 3$, our example in Section \ref{Sec counterexample} also shows that the Ramanadov Conjecture for the Bergman kernel fails for higher dimensional normal Stein spaces. Recall that the Ramadanov Conjecture (c.f., \cite{Ramadanov1981}, \cite[Question 3]{EnZh}) proposes that if the logarithmic term in Fefferman's asymptotic expansion \cite{Fe2} of the Bergman kernel vanishes to infinite order at the boundary of a normal reduced Stein space with compact, smooth strongly pseudoconvex boundary, then the boundary is spherical. The Ramadanov Conjecture has been established in two dimensions by the work of D. Burns and R. C. Graham (see \cite{Graham1987b}). The normal reduced Stein space constructed in Section \ref{Sec counterexample} gives a $3$-dimensional counterexample with one isolated singularity. The counterexamples in \cite{EnZh} are smooth, but not Stein.
\end{rmk}

Theorem \ref{main theorem intro} has two immediate consequences in the non-singular case:
\begin{cor}\label{main corollary intro}
	Let $V$ be a $2$-dimensional algebraic variety in $\mathbb{C}^N$, and let $G$ be a relatively compact domain in $V$ with smooth strongly pseudoconvex boundary. Assume that every point in $\overline{G}$ is a smooth point of $V$. Then the Bergman kernel form of $G$ is algebraic if and only if $G$ is biholomorphic to $\mathbb{B}^2$ by an algebraic map.
\end{cor}

\begin{cor}\label{main corollary 2 intro}
	Let $G$ be a bounded domain in $\mathbb{C}^2$ with smooth strongly pseudoconvex boundary. Then the Bergman kernel of $G$ is rational (respectively, algebraic) if and only if there is a rational (respectively, algebraic) biholomorphic map from $G$ to $\mathbb{B}^2$.
\end{cor}

We remark that although Theorem \ref{main theorem intro} fails in higher dimension, Corollary \ref{main corollary intro} and \ref{main corollary 2 intro} might still be true. For instance, it is clear from the proof below of Theorem \ref{main theorem intro} (see Remark \ref{rmk:Ramadanov}) that if the Ramadanov Conjecture is proved to hold for, e.g., strongly pseudoconvex bounded domains in $\mathbb{C}^n$, which is still a possibility despite Remark \ref{rmk counterexample intro} above, then Corollary \ref{main corollary 2 intro} also holds in $\bC^n$.

We also remark that the rationality of the biholomorphic map $G\to \bB^2$ in Corollary \ref{main corollary 2 intro}, once its existence has been established, follows from the work of S. Bell \cite{Bell}. For the reader's convenience, a self-contained proof of the rationality is given in Section \ref{Sec Algebraic BK}.

As a final remark in this introduction, we note that, by Lempert's algebraic approximation theorem \cite{Le}, if $G$ is a relatively compact domain in a reduced Stein space $X$ with only isolated singularities, then there exist an affine algebraic variety $V$, a domain $\Omega \subset V$, and a biholomorphism $F$ from a neighborhood of $\oo{G}$ to a neighborhood of $\oo{\Omega}$ with $F(\Omega)=G$. We shall say such a domain $\Omega$ is an \emph{algebraic realization} of $G$. Theorem \ref{main theorem intro} implies the following corollary.

\begin{cor} Let $G$ be a relatively compact domain in a $2$-dimensional reduced Stein space $X$ with smooth strongly pseudoconvex boundary and only isolated normal singularities. If $G$ has an algebraic realization with an algebraic Bergman kernel, then $G$ is biholomorphic to a ball quotient $\mathbb{B}^2/\Gamma$, where $\Gamma$ is a finite unitary group with no fixed point on $\partial \mathbb{B}^2$.

\end{cor}

To prove the "only if" implication in Theorem \ref{main theorem intro}, we use the asymptotic boundary behavior of the Bergman kernel to establish algebraicity and sphericity of the boundary of $G$.  Fefferman's asymptotic expansion \cite{Fe2} and the Riemann mapping type theorems due to X. Huang--S. Ji (\cite{HuJi98}) and X. Huang (\cite{Hu}) play important roles in the proof. To prove the converse ("if") implication in the theorem, we will need to compute the Bergman kernel forms of finite ball quotients. In order to do so, we shall establish a transformation formula for (possibly branched) covering maps of complex analytic spaces. This formula generalizes a classical theorem of Bell (\cite{Bell81},  \cite{Bell82}):


\begin{thm}\label{BK transformation thm}
Let $M_1$  and $M_2$ be two complex analytic sets.
Let $V_1\subset M_1$ and $V_2 \subset M_2$ be proper analytic subvarieties such that $M_1-V_1, M_2-V_2$ are complex manifolds of the same dimension.
Assume that $f:M_1-V_1\rightarrow M_2-V_2$ is a finite ($m-$sheeted) holomorphic covering map. Let $\Gamma$ be the deck transformation group for the covering map (with $|\Gamma|=m$), and denote by $K_i(z,\bar{w})$ the Bergman kernels of $M_i$ for $i=1,2$. Then the Bergman kernel forms transform according to
	\begin{equation}\label{BK transformation eq}
	\sum_{\gamma\in \Gamma}(\gamma, \id)^*K_1=\sum_{\gamma\in \Gamma}(\id, \gamma)^*K_1=(f,f)^*K_2 \quad~\text{on}~(M_1-V_1) \times (M_1-V_1),
	\end{equation}
	where $\id: M_1\rightarrow M_1$ is the identity map.
\end{thm}

See Section 2 for the notation used in the formula in Theorem \ref{BK transformation thm}. We expect that this formula will be useful in other applications as well. In an upcoming paper \cite{EbenfeltXiaoXu2020}, the authors apply it to study the question of when the Bergman metric of a finite ball quotient is K\"ahler--Einstein. (This is always the case for finite disk quotients, i.e., one-dimensional ball quotients, by recent work of X. Huang and X. Li \cite{HuLi}.)

The paper is organized as follows. Section \ref{Sec background} gives some preliminaries on algebraic functions and Bergman kernels of complex analytic spaces. Section \ref{Sec transformation law} is devoted to establishing the transformation formula in Theorem \ref{BK transformation thm}. Then in Section \ref{Sec Bergman kernel for finite ball quotient} we apply it to show that every standard algebraic realization (in particular, Cartan's canonical realization) of a finite ball quotient must have algebraic Bergman kernel, and thus prove the "if" implication in Theorem \ref{main theorem intro}. Section \ref{Sec Algebraic BK} gives the proof of the "only if" implication in Theorem \ref{main theorem intro}, as well as those of Corollaries \ref{main corollary intro} and \ref{main corollary 2 intro}. In Section \ref{Sec counterexample} and Appendix \ref{Sec Appendix}, we construct the counterexample mentioned above to the corresponding statement of Theorem \ref{main theorem intro} in higher dimensions.

{\bf Acknowledgment.} The second author thanks Xiaojun Huang for many inspiring conversations on quotient singularities.

\section{Preliminaries}\label{Sec background}

\subsection{Algebraic Functions}

In this subsection, we will review some basic facts about algebraic functions. For more details, we refer the readers to \cite[Chapter 5.4]{BER} and \cite{HuJi02}.

\begin{defn}[Algebraic functions and maps]
Let $\mathbb{K}$ be the field $\mathbb{R}$ or $\mathbb{C}$. Let $U\subset \mathbb{K}^n$ be a domain. A $\mathbb{K}-$analytic function $f: U\rightarrow \mathbb{K}$ is said to be $\mathbb{K}-$algebraic (i.e., real/complex-algebraic) on $U$ if there is a non-trivial polynomial $P(x,y)\in \mathbb{K}[x,y]$, with $(x,y)\in \mathbb{K}^n\times \mathbb{K}$, such that $P(x,f(x))=0$ for all $x\in U$.
We say that a $\mathbb{K}-$analytic map $F: U \rightarrow \mathbb{C}^N$ is $\mathbb{K}-$algebraic if each of its components is so on $U$.
\end{defn}

\begin{rmk}\label{rmk:RvsC} We make two remarks:
\begin{itemize}
\item[(i)] If $f(x)$ is an $\bK$-analytic function in a domain $U\subset\bK^n$, then $f$ is $\bK$-algebraic if and only if it is $\bK$-algebraic in some neighborhood of any point $x_0\in U$.
\item[(ii)] If $f(x)$ is an $\bR$-analytic function in a domain $U\subset\bR^n$, then there is domain $\hat U\subset \bC^n$ containing $U\subset\bR^n\subset\bC^n$ and a $\bC$-analytic (i.e., holomorphic) function $g(x+iy)$ in $\hat U$ such that $f=g|_U$; i.e., $f(x)=g(x)$ for $x\in U$. Moreover, $f$ is $\bR$-algebraic if and only if $g$ is $\bC$-algebraic.
\end{itemize}
\end{rmk}




We say a differential form on $U\subset\mathbb{C}^n\cong\mathbb{R}^{2n}$ is real-algebraic  if each of its coefficient functions is so. We can also define real-algebraicity of a differential form on an affine (algebraic) variety.

\begin{defn}\label{algebraic on affine variety def}
Let $V\subset \mathbb{C}^N$ be an affine variety and write $\Reg V$ for the set of its regular points. Let $\phi$ be a real analytic differential form on $\Reg V$.
We say $\phi$ is real-algebraic on $V$ if for every point $z_0\in \Reg V$, there exists a real-algebraic differential form $\psi$ in a neighborhood $U$ of $z_0$ in $\mathbb{C}^N\cong\mathbb{R}^{2N}$ such that
	\begin{equation*}
		\psi|_V=\phi, \quad \mbox{ on } U\cap V.
	\end{equation*}
\end{defn}

Let $T_{z_0}V\cong T^{1,0}_{z_0}V$ be the complex tangent space of $V$ at a smooth point $z_0\in V$ considered as an affine complex subspace in $\bC^n$ through $z_0$, and let $\xi=(\xi_1,\cdots,\xi_n)$ be affine coordinates for $T_{z_0}V$. Since $V$ can be realized locally as a graph over $T_{z_0}V$, the real and imaginary parts of $\xi$ also serve as local real coordinates for $V$ near $z_0$. We call such coordinates the \emph{canonical extrinsic coordinates at $z_0$}. Then the following statements are equivalent.

\begin{itemize}
	\item[(a)] $\phi$ is real-algebraic on $\Reg V$ (in the sense of Definition \ref{algebraic on affine variety def}).
	\item[(b)] For any $z_0\in \Reg V$, $\phi$ is real-algebraic in canonical extrinsic coordinates at $z_0$.
\end{itemize}

If in addition, there is a domain $G\subset \mathbb{C}^n$ and a $\bC$-algebraic (i.e., holomorphic algebraic) immersion $f: G\rightarrow \mathbb{C}^N$ such that $f(G)=\Reg V$, then (a) and (b) are further equivalent to
\begin{itemize}
	\item[(c)] $f^*\phi$ is real-algebraic on $G$.
\end{itemize}

\begin{rmk}
We can define complex-algebraicity of $(p,0)-$forms, $p>0,$
on an complex affine (algebraic) variety in a similar manner as in Definition \ref{algebraic on affine variety def}.
\end{rmk}

\subsection{The Bergman Kernel}
In this section, we will briefly review some properties of the Bergman kernel on a complex manifold. More details can be found in \cite{KoNo}.

Let $M$ be an n-dimensional complex manifold. Write $L^2_{(n,0)}(M)$ for the space of $L^2$-integrable $(n,0)$ forms on $M,$  which is equipped with the following inner product:
\begin{equation}\label{inner product}
(\varphi,\psi)_{L^2(M)}:=i^{n^2}\int_{M}\varphi\wedge\oo{\psi},
\quad \varphi,\psi \in L^2_{(n,0)}(M),
\end{equation}

Define the {\em Bergman space} of $M$  to be
\begin{equation}\label{Bergman space of forms}
A^2_{(n,0)}(M):=\bigl\{\varphi \in L^2_{(n,0)}(M): \varphi \mbox{ is a holomorphic $(n,0)$ form on $M$} \}.
\end{equation}

Assume $A^2_{(n,0)}(M) \neq \{0\}$. Then $A^2_{(n,0)}(M)$ is a separable Hilbert space. Taking any orthonormal basis $\{\varphi_k\}_{k=1}^{q}$ of $A^2_{(n,0)}(M)$ (here $1 \leq q \leq \infty$), we define the {\em Bergman kernel (form)} of $M$ to be
\begin{equation*}
K_M(x,\bar{y})=i^{n^2}\sum_{k=1}^{q} \varphi_k(x)\wedge \oo{\varphi_k(y)}.
\end{equation*}
Then, $K_M(x, \bar{x})$ is a real-valued, real analytic form of degree $(n,n)$ on $M$ and is independent of the choice of orthonormal basis. When $M$ is also (the set of regular points on) an affine variety,  we say that the Bergman kernel of $M$ is {\em algebraic} if $K_{M}(x, \bar{x})$ is real-algebraic in the sense of Definition \ref{algebraic on affine variety def}.  The following definitions and facts are standard in literature.

\begin{defn}[Bergman projection]
	Given $g\in L^2_{(n,0)}(M)$, we define for $x\in M$
	\begin{equation*}
		Pg(x)=\int_Mg(\zeta)\wedge K_M(x,\bar{\zeta}):=i^{n^2}\sum_{k=1}^{q}\Bigl(\int_Mg(\zeta)\wedge\oo{\varphi_k(\zeta)}\Bigr)\varphi_k(x).
	\end{equation*}
\end{defn}
$P\colon L^2_{(n,0)}\to A^2_{(n,0)}(M)$ is called the {\em Bergman projection}, and is the orthogonal projection to the Bergman space $A^2_{(n,0)}(M)$.

The Bergman kernel form remains unchanged if we remove a proper complex analytic subvariety. The following theorem is from \cite{Ko}.
\begin{thm}[\cite{Ko}]\label{Kobayashi thm}
	If $M'$ is a domain in an $n$-dimensional complex manifold $M$ and if $M-M'$ is a complex analytic subvariety of $M$ of complex dimension $\leq n-1$, then
	\begin{equation*}
		K_M(x,\bar{y})=K_{M'}(x,\bar{y}) \quad \mbox{ for any y }\in M'.
	\end{equation*}
\end{thm}

This theorem suggests the following generalization of the Bergman kernel form to complex analytic spaces.
\begin{defn}
Let $M$ be a reduced complex analytic space, and let $V\subset M$ denote its set of singular points. The Bergman kernel form of $M$ is defined as
	\begin{equation*}
		K_M(x,\bar{y})=K_{M-V}(x,\bar{y})\quad \mbox{ for any } x, y\in M-V,
	\end{equation*}
	where $K_{M-V}$ denotes the Bergman kernel form of the complex manifold consisting of regular points of $M$.
\end{defn}

Let $N_1, N_2$ be two complex manifolds of dimension $n$. Let $\gamma: N_1\rightarrow M$ and $\tau: N_2\rightarrow M$ be holomorphic maps. The pullback of the Bergman kernel $K_M(x,\bar{y})$ of $M$ to $N_1 \times N_2$ is defined in the standard way. That is, for any $z\in N_1, w\in N_2$,
\begin{equation*}
	\bigl( (\gamma,\tau)^*K\bigr)(z,\bar{w})=\sum_{k=1}^{q} \gamma^*\varphi_k(z)\wedge \oo{\tau^*\varphi_k(w)}.
\end{equation*}
In terms of local coordinates, writing the Bergman kernel form of $M$ as
\begin{equation*}
	K_M(x,\bar{y})=\widetilde{K}(x,\bar{y})dx_1\wedge\cdots dx_n\wedge d\oo{y_1}\wedge\cdots \wedge d\oo{y_n},
\end{equation*}
we have
\begin{equation*}
	\bigl( (\gamma,\tau)^*K_M\bigr)(z,\bar{w})=\widetilde{K}(\gamma(z),\oo{\tau(w)})\,J_{\gamma}(z)\,\oo{J_{\tau}(w)}\,dz_1\wedge\cdots dz_n\wedge d\oo{w_1}\wedge\cdots \wedge d\oo{w_n},
\end{equation*}
where $J_{\gamma}$ and $J_{\tau}$ are the Jacobian determinants of the maps $\gamma$ and $\tau$, respectively.

\section{The transformation law for the Bergman kernel}\label{Sec transformation law}

In this section, we shall prove Theorem \ref{BK transformation thm}.
For this, we shall adapt the ideas in \cite{Bell82} to our situation. More precisely, we shall first prove the following transformation law for the Bergman projections. Then \eqref{BK transformation eq} will follow readily by comparing the associated distributional kernels for the projection operators.

\begin{prop}\label{BP transformation prop}
	Under the assumptions and notation in Theorem $\ref{BK transformation thm}$, we denote by $n$ the complex dimension of $M_1-V_1$ and $M_2-V_2$. Let $P_i: L^2_{(n,0)}(M_i-V_i)\rightarrow A^2_{(n,0)}(M_i-V_i)$  e the Bergman projection for $i=1,2$. Then the Bergman projections transform according to
	\begin{equation}\label{BP transformation eq}
		P_1(f^*\phi)=f^*(P_2\phi) \quad \mbox{ for any } \phi\in L^2_{(n,0)}(M_2-V_2).
	\end{equation}
\end{prop}

We first check that $f^*\phi\in L^2_{(n,0)}(M_1-V_1)$ if $\phi\in L^2_{(n,0)}(M_2-V_2)$ in the next lemma. Recall that $f$ is an $m$-sheeted covering map $M_1-V_1\rightarrow M_2-V_2$.
\begin{lemma}\label{integral between cover and base I lemma}
	\begin{equation*}
		\|f^*\phi\|_{L^2(M_1-V_1)}=m^{\frac{1}{2}}\|\phi\|_{L^2(M_2-V_2)} \quad \mbox{ for any } \phi\in L^2_{(n,0)}(M_2-V_2).
	\end{equation*}
\end{lemma}

\begin{proof}
	Let $\{U_j\}$ be a countable, locally finite open cover of $M_2-V_2$ such that
	\begin{itemize}
		\item  each $U_j$ is relatively compact;
		\item  $f^{-1}(U_j)=\cup_{k=1}^m V_{j,k}$ for some pairwise disjoint open sets $\{V_{j,k}\}_{k=1}^m$ on $M_1-V_1$;
		\item $f:V_{j, k}\rightarrow U_j$ is a biholomorphsm for each $j=1,2,\cdots m$.
	\end{itemize}
	Let $\{\rho_j\}$ be a partition of unity subordinate to the cover $\{U_j\}$. Then
	\begin{align*}
		i^{n^2}\int_{M_2-V_2} \phi\wedge\oo{\phi}
		=\sum_ji^{n^2}\int_{U_j}\rho_j \phi\wedge\oo{\phi}=\frac{1}{m} \sum_j\sum_{k=1}^m i^{n^2}\int_{V_{j,k}} (f^*\rho_j)\, f^*\phi\wedge\oo{f^*\phi}.
	\end{align*}
	Note that $\{f^*\rho_j\}$ is a partition of unity subordinate to the countable, locally finite open cover $\{\cup_{k=1}^m V_{j,k}\}$ of $M_1-V_1$. Thus,
	\begin{align*}
		\frac{1}{m} \sum_j\sum_{k=1}^m i^{n^2}\int_{V_{j,k}} (f^*\rho_j)\, f^*\phi\wedge\oo{f^*\phi}
		=&\frac{1}{m} \sum_j i^{n^2}\int_{\cup_{k=1}^mV_{j,k}} (f^*\rho_j)\, f^*\phi\wedge\oo{f^*\phi}
		\\
		=&\frac{1}{m} i^{n^2}\int_{M_1-V_1} f^*\phi\wedge\oo{f^*\phi}.
	\end{align*}
The result therefore follows immediately.
\end{proof}

Let $F_1, F_2, \cdots, F_m$ be the $m$ local inverses to $f$ defined locally on $M_2-V_2$. Note that $\sum_{k=1}^mF_k^*$ is a well-defined operator on $L_{(n,0)}^2 (M_1-V_1)$, though each individual $F_k$ is only locally defined.
\begin{lemma}\label{integral between cover and base II lemma}
	Let $v\in L^2_{(n,0)}(M_1-V_1)$ and $\phi\in L^2_{(n,0)}(M_2-V_2)$. Then $\sum_{k=1}^m F_k^*(v) \in L^2_{(n,0)}(M_2-V_2)$ and
	\begin{equation}\label{integral between cover and base II eq}
		\bigl(v,f^*\phi\bigr)_{L^2(M_1-V_1)}=\bigl(\sum_{k=1}^m F_k^*(v), \phi\bigr)_{L^2(M_2-V_2)}.
	\end{equation}
\end{lemma}
\begin{proof}
We first verify $\sum_{k=1}^m F_k^*(v) \in L^2_{(n,0)}(M_2-V_2)$. For that we note
\begin{equation*}
	f^*\sum_{k=1}^m F_k^*(v)=\sum_{\gamma\in \Gamma}\gamma^*v.
\end{equation*}
By the same argument as in Lemma \ref{integral between cover and base I lemma}, we have
\begin{equation}\label{eqlemma34}
	\bigl\|\sum_{\gamma\in \Gamma}\gamma^*v\bigr\|_{L^2(M_1-V_1)}=m^{\frac{1}{2}}\bigl\|\sum_{k=1}^m F_k^*(v)\bigr\|_{L^2(M_2-V_2)}.
\end{equation}
Since each deck transformation $\gamma: M_1-V_1\rightarrow M_1-V_1$ is biholomorphic, it follows that
\begin{align*}
	\bigl\|\sum_{\gamma\in \Gamma}\gamma^*v\bigr\|_{L^2(M_1-V_1)}
	\leq \sum_{\gamma\in \Gamma}\bigl\|\gamma^*v\bigr\|_{L^2(M_1-V_1)}=m\|v\|_{L^2(M_1-V_1)}.
\end{align*}
Therefore by \eqref{eqlemma34}, $\sum_{k=1}^m F_k^*(v) \in L_{(n,0)}^2(M_2-V_2)$.

Now we are ready to prove \eqref{integral between cover and base II eq}. Let $\{U_j\}$, $\{V_{j,k}\}$ and $\{\rho_j\}$ be the open covers and partition of unity as in Lemma \ref{integral between cover and base I lemma}. Then
\begin{align*}
	\bigl(\sum_{k=1}^m F_k^*(v), \phi\bigr)_{L^2(M_2-V_2)}=\sum_{j}i^{n^2}\int_{U_j}\rho_j\sum_{k=1}^m F_k^*(v)\wedge \oo{\phi}.
\end{align*}
Note that every $F_k: U_j\rightarrow V_{j,k}$ is biholomorphic and the inverse of $f: V_{j,k}\rightarrow U_j$. Thus,
\begin{align*}
	\sum_{j}i^{n^2}\int_{U_j}\rho_j\sum_{k=1}^m F_k^*(v)\wedge \oo{\phi}=\sum_{j}\sum_{k=1}^mi^{n^2}\int_{V{j,k}}(f^*\rho_j) v\wedge \oo{f^*\phi}=(v,f^*\phi)_{L^2(M_1-V_1)}.
\end{align*}
The last equality follows from the fact that $\{f^*\rho_j\}$ is a partition of unity subordinate to the countable, locally finite open cover $\{\cup_{k=1}^m V_{j,k}\}$ of $M_1-V_1$. This proves \eqref{integral between cover and base II eq}.
\end{proof}

We are now ready to prove Proposition \ref{BP transformation prop}.
\begin{proof}[Proof of Proposition \ref{BP transformation prop}]
If $\phi\in A^2_{(n,0)}(M_2-V_2)$, then $f^*\phi\in A^2_{(n,0)}(M_1-V_1)$ by Lemma \ref{integral between cover and base I lemma}, whence \eqref{BP transformation eq} holds trivially. It thus suffices to prove \eqref{BP transformation eq} for $\phi\in A^2_{(n,0)}(M_2-V_2)^{\perp}$. In this case, \eqref{BP transformation eq} reduces to
\begin{equation*}
	P_1(f^*\phi)=0 \quad \mbox{ for any } \phi\in A^2_{(n,0)}(M_2-V_2)^{\perp};
\end{equation*}
i.e., $\phi\in A^2_{(n,0)}(M_2-V_2)^{\perp}$ implies that $f^*\phi\in A^2_{(n,0)}(M_1-V_1)^{\perp}$. To prove this, we note that for any $v\in A^2_{(n,0)}(M_1-V_1)$, we have by Lemma \ref{integral between cover and base II lemma}
\begin{align*}
	\bigl(v,f^*\phi)_{L^2(M_1-V_1)}=\bigl(\sum_{k=1}^m F_k^*(v), \phi\bigr)_{L^2(M_2-V_2)}=0.
\end{align*}
The last equality follows from the fact $\phi\in
 A^2_{(n,0)}(M_2-V_2)^{\perp}$. Thus, $f^*\phi\in A^2_{(n,0)}(M_1-V_1)^{\perp}$ and the proof is completed.
\end{proof}

We are now in a position to prove Theorem \ref{BK transformation thm}.
\begin{proof}[Proof of Theorem \ref{BK transformation thm}]
Let $\id_{M_i}$ be the identity map on $M_i$ for $i=1,2$. Recall that $\{F_k\}_{k=1}^m$ are local inverses of $f$. Note that $\sum_{k=1}^m(\id_{M_1}, F_k)^*K_1$ is a well-defined $(n,n)$ form on $(M_1-V_1)\times (M_2-V_2)$ though each $(\id_{M_1}, F_k)^*K_1$ is only locally defined.

We shall write out the Bergman projection transformation law \eqref{BP transformation eq} in terms of integrals of the Bergman kernel forms. For any $\phi\in L^2_{(n,0)}(M_2-V_2)$, by Lemma \ref{integral between cover and base II lemma} we have for any $z\in M_1-V_1$,
\begin{align*}
	P_1(f^*\phi)(z)=\int_{M_1-V_1} f^*\phi(\eta)\wedge K_1(z,\eta)=\int_{M_2-V_2}\phi(\eta)\wedge \sum_{k=1}^m(\id_{M_1}, F_k)^*K_1(z,\eta).
\end{align*}
On the other hand,
\begin{align*}
	P_2(\phi)(\xi)=\int_{M_2-V_2}\phi(\eta)\wedge K_2(\xi,\eta) \quad \mbox{ for any }\xi \in M_2-V_2.
\end{align*}
If we pull back the forms on both sides by $f$, then
\begin{align*}
	f^*P_2(\phi)(z)=\int_{M_2-V_2}\phi(\eta)\wedge (f,\id_{M_2})^*K_2(z,\eta) \quad \mbox{ for any }z \in M_1-V_1.
\end{align*}
Therefore, the Bergman projection transformation law \eqref{BP transformation eq} translates to
\begin{align*}
	\int_{M_2-V_2}\phi(\eta)\wedge \sum_{k=1}^m(\id_{M_1}, F_k)^*K_1(z,\eta)
	=\int_{M_2-V_2}\phi(\eta)\wedge (f,\id_{M_2})^*K_2(z,\eta).
\end{align*}
As this equality holds for any $\phi\in L^2_{(n,0)}(M_2-V_2)$, it follows that for any $z \in M_1-V_1$ and $\eta\in M_2-V_2$,
\begin{equation}
	\sum_{k=1}^m(\id_{M_1}, F_k)^*K_1(z,\eta)
	=(f,\id_{M_2})^*K_2(z,\eta).
\end{equation}
If we further pull back the forms on both sides by $(\id_{M_1},f): (M_1-V_1) \times (M_1-V_1)\rightarrow (M_1-V_1)\times (M_2-V_2)$, then we obtain for $z,w\in M_1-V_1$,
\begin{align}
	\sum_{k=1}^m(\id_{M_1}, F_k\circ f)^*K_1(z,w)
	=(f,f)^*K_2(z,w).
\end{align}
By using the notation $\gamma_k$ for the deck transformation $F_k\circ f$, we may write this as
\begin{align}\label{transformation law 1}
	\sum_{k=1}^m(\id_{M_1}, \gamma_k)^*K_1(z,w)
	=(f,f)^*K_2(z,w).
\end{align}

Note that
\begin{align*}
	\sum_{k=1}^m(\id_{M_1}, \gamma_k)^*K_1(z,w)=\sum_{k=1}^m(\gamma_k\circ \gamma_k^{-1}, \gamma_k\circ \id_{M_1})^*K_1(z,w)=\sum_{k=1}^m(\gamma_k^{-1}, \id_{M_1})^* (\gamma_k,\gamma_k)^*K_1(z,w).
\end{align*}
Since $\gamma_k$ is a biholomorphism on $M_1-V_1$, we have
\begin{equation*}
	(\gamma_k,\gamma_k)^*K_1(z,w)=K_1(z,w),
\end{equation*}
and hence
\begin{align*}
	\sum_{k=1}^m(\id_{M_1}, \gamma_k)^*K_1(z,w)=\sum_{k=1}^m(\gamma_k^{-1}, \id_{M_1})^* K_1(z,w)=\sum_{k=1}^m(\gamma_k, \id_{M_1})^* K_1(z,w).
\end{align*}
Theorem \ref{BK transformation thm} now follows by combining the above identity with \eqref{transformation law 1}.
\end{proof}

\section{Proof of Theorem \ref{main theorem intro}, part I:  Bergman kernels of ball quotients}
\label{Sec Bergman kernel for finite ball quotient}
In this section, we will apply the transformation law in Theorem \ref{BK transformation thm} to study the Bergman kernel form of a finite ball quotient and prove the  "if" implication in Theorem \ref{main theorem intro}. For this part, the restriction of the dimension of the algebraic variety to two is not needed, and we shall therefore consider the situation in an arbitrary dimension $n$.

Let $\mathbb{B}^n$ denote the unit ball in $\mathbb{C}^n$ and $\text{Aut}(\mathbb{B}^n)$ its (biholomorphic) automorphism group.
Let $\Gamma$ be a finite subgroup of $\text{Aut}(\mathbb{B}^n)$. As the unitary group $U(n)$ is a maximal compact subgroup of $\Aut(\mathbb{B}^n)$, by basic Lie group theory, there exists some $\psi\in \Aut(\mathbb{B}^n)$ such that $\Gamma\subset \psi^{-1}\cdot U(n)\cdot \psi$. Thus without loss of generality, we can assume $\Gamma\subset U(n)$, i.e., $\Gamma$ is a finite unitary group. Note that the origin $0\in \mathbb{C}^n$ is always a fixed point of every element in $\Gamma$. We say $\Gamma$ is \emph{fixed point free} if every $\gamma\in \Gamma-\{\id\}$ has no other fixed point, or equivalently, if every $\gamma\in \Gamma-\{\id\}$ has no fixed point on $\partial\mathbb{B}^n$. In this case, the action of $\Gamma$ on $\partial\mathbb{B}^n$ is properly discontinuous and $\partial\mathbb{B}^n/\Gamma$ is a smooth manifold.

By a theorem of Cartan \cite{Cartan}, the quotient $\mathbb{C}^n/\Gamma$ can be realized as a normal algebraic subvariety $V$ in some $\mathbb{C}^N$. To be more precise, we write $\mathcal{A}$ for the algebra of $\Gamma$ invariant holomorphic polynomials, that is,
\begin{equation*}
	\mathcal{A}:=\big\{p\in \mathbb{C}[z_1,\cdots, z_n]: p\circ\gamma=p \,\mbox{ for all }\gamma \in\Gamma \big\}.
\end{equation*}
By Hilbert's basis theorem, $\mathcal{A}$ is finitely generated. Moreover, we can find a minimal set of homogeneous polynomials $\{p_1, \cdots, p_N\}$ such that every $p\in \mathcal{A}$ can be expressed in the form
\begin{equation*}
	p(z)=q(p_1(z), \cdots, p_N(z)) \quad \mbox{ for }z\in \mathbb{C}^n,
\end{equation*}
where $q$ is some holomorphic polynomial in $\mathbb{C}^N$.
The map $Q:=(p_1,\cdots,p_N): \mathbb{C}^n\rightarrow \mathbb{C}^N$ is proper and induces a homeomorphism of $\mathbb{C}^n/\Gamma$ onto $V:=Q(\mathbb{C}^n)$. By Remmert's proper mapping theorem (see \cite{GH}), $V$ is an analytic variety. As $Q$ is a polynomial holomorphic map, $V$ is furthermore an algebraic variety. The restriction of $Q$ to the unit ball $\mathbb{B}^n$ maps $\mathbb{B}^n$ properly onto a relatively compact domain $G\subset V$. In this way, $\mathbb{B}^n/\Gamma$ is realized as $G$ by $Q$.
Following \cite{Rudin}, we call such $Q$ the \emph{basic map} associated to $\Gamma$.
The ball quotient $G=\mathbb{B}^n/\Gamma$ is nonsingular if and only if the group $\Gamma$ is generated by \emph{reflections}, i.e., elements of finite order in $U(n)$ that fix a complex subspace of dimension $n-1$ in $\mathbb{C}^n$ (see \cite{Rudin}); thus, if $\Gamma$ is fixed point free and nontrivial, then $G=\mathbb{B}^n/\Gamma$ must have singularities.
 Moreover, $G$ has smooth boundary if and only if $\Gamma$ is fixed point free (see \cite{Fo86} for more results along this line).

We are now in a position to state the following theorem, which implies the "if" implication in Theorem \ref{main theorem intro}.
\begin{thm} \label{main theorem backward direction}
Let $G$ be a domain in an algebraic variety $V$ in  $\mathbb{C}^N$ and $\Gamma \subset U(n)$  a finite  unitary subgroup with $|\Gamma|=m$. Suppose
there exist proper complex analytic varieties $V_1\subset \mathbb{B}^n$, $V_2 \subset G$ and $F: \mathbb{B}^n-V_1\rightarrow G-V_2$ such that $F$ is an m-sheeted covering map with deck transformation group $\Gamma$. If $F$ is algebraic, then the Bergman kernel form of $G$ is algebraic.
\end{thm}

\begin{proof}
Note that the Bergman kernel form of $G$ coincides with that of $\widetilde{G}:=G-V_2$ by Theorem \ref{Kobayashi thm}, and likewise the Bergman kernel form $K_{\mathbb{B}^n}$ of $\mathbb{B}^n$ coincides with that of $\widetilde{B}:=\mathbb{B}^n-V_1$. By the transformation law in Theorem \ref{BK transformation thm}, we have
	\begin{equation*}
		\sum_{\gamma\in\Gamma}(\id_{\mathbb{B}^n},\gamma)^*K_{\mathbb{B}^n}=(F,F)^*K_G \quad \mbox{ on } \widetilde{B}\times\widetilde{B}.
	\end{equation*}
	Since all $\gamma\in \Gamma$ and $K_{\mathbb{B}^n}$ are rational,  so is the right hand side of the equation. This implies that $K_G$ is algebraic (see the equivalent condition (c) of algebraicity in \S 2.1).
\end{proof}


Theorem \ref{main theorem backward direction} applies in particular to Cartan's canonical realization of ball quotient.

\begin{cor}\
Let $\Gamma\subset U(n)$ be a finite unitary group. Suppose $Q:\mathbb{C}^n\rightarrow \mathbb{C}^N$ is the basic map associated to $\Gamma$. Let $G=Q(\mathbb{B}^n)$, which is a relatively compact domain in the algebraic variety $V=Q(\mathbb{C}^n)$.
Then the Bergman kernel form of $G$ is algebraic.
\end{cor}

\begin{proof}
We let
\begin{equation*}
		Z=\{z\in \mathbb{C}^n: \mbox{the Jacobian of $Q$ at $z$ is not full rank}  \}.
	\end{equation*}
Clearly, $Z$ is a proper complex analytic variety in $\mathbb{C}^n$. By Remmert's proper mapping theorem, $Q(Z)\subset V$ is a proper complex analytic variety. Moreover, $Q:\mathbb{B}^n-Z\rightarrow G-Q(Z)$ is a covering map with $m$ sheets, where $m=|\Gamma|$, and  $\Gamma$ is its deck transformation group (Note that $Q^{-1}(Q(Z))=Z$; see \cite{Cartan}). The conclusion now follows from Theorem \ref{main theorem backward direction}.
\end{proof}

\begin{rmk}
Note that the "if" implication in Theorem \ref{main theorem intro} in fact holds under a much weaker assumption than that stipulated in the theorem.  In Theorem \ref{main theorem backward direction} we do not assume $n=2$ nor that the group $\Gamma$ is fixed point free. {We remark that the formula for the Bergman kernel of the finite ball quotient is also obtained by Huang-Li \cite{HuLi}.}
\end{rmk}

\section{Proof of Theorem \ref{main theorem intro}, part II}\label{Sec Algebraic BK}
In this section, we prove one of the main results of the paper---the "only if" implication in Theorem \ref{main theorem intro}. We also prove Corollary \ref{main corollary intro} and \ref{main corollary 2 intro}. 

\begin{proof}[Proof of the "only if" implication in Theorem \ref{main theorem intro}.]
Let $V$ and $G$ be as in Theorem \ref{main theorem intro} and assume that $G$ has algebraic Bergman kernel.  We shall prove that $G$ is a finite ball quotient. We proceed in several steps.
	
	\textbf{Step 1.} In this step, we prove $\partial G$ is real analytic, and furthermore, real algebraic. For this step, we do not need to assume that the dimension of $V$ is two.
	\begin{prop}\label{boundary algebraic prop}
		Let $G$ be a relatively compact domain in an $n$-dimensional ($n\geq 2$) algebraic variety $V\subset\mathbb{C}^N$ with smooth strongly pseudoconvex boundary. If the Bergman kernel $K_G$ of $G$ is algebraic, then the boundary $\partial G$ of $G$ is Nash algebraic, i.e., $\partial G$ is locally defined by a real algebraic function.
	\end{prop}
	\begin{proof}
		Fix a point $p\in \partial G$. Then there exists a neighborhood $U$ of $p$ in $V$ with canonical extrinsic coordinates
$z=(z_1,\cdots, z_n)$ on $U$ (see Section 2). Write the Bergman kernel form $K_G$ of $G$ as
		\begin{equation*}
		K_G=K(z,\bar{z})dz\wedge d\oo{z} \quad \mbox{ on } U\cap G,
		\end{equation*}
		where $dz=dz_1\wedge\cdots \wedge dz_n$, $d\oo{z}=d\oo{z_1}\wedge\cdots \wedge d\oo{z_n}$ and $K(z,\bar{z})$ is a real algebraic function on $U\cap G$.

As $K$ is real algebraic, there exist real-valued polynomials $a_1(z,\bar{z}), \cdots, a_q(z,\bar{z})$ in $\mathbb{C}^n\cong\mathbb{R}^{2n}$ with $a_q\neq 0$ such that
\begin{equation}\label{eqnkz}
		a_q(z,\bar{z})K(z,\bar{z})^q+\cdots+a_1(z,\bar{z})K(z,\bar{z})+a_0(z,\bar{z})=0, \quad \mbox{on } U\cap G.
\end{equation}
Note that when  $z\rightarrow \partial G$, we have $K(z,\bar{z})\rightarrow \infty$ as $\partial G$ is strictly pseudoconvex.
We divide both sides of \eqref{eqnkz}  by $K(z,\bar{z})^q$ and let $z \rightarrow \partial G$ to obtain
\begin{equation*}
		a_q(z,\bar{z})=0, \quad \mbox{ on } U\cap \partial G.
\end{equation*}
		
		Write $z_k=x_k+iy_k$ for $1\leq k\leq n$, $z'=(z_1,\cdots, z_{n-1})$ and $x'=(x_1,y_1,\cdots, x_{n-1},y_{n-1},x_n)$. By rotation, we can assume that $\partial G$ near $p$ is locally defined by
		\begin{equation*}
		y_n=\varphi(x'),
		\end{equation*}
where $\varphi$ is a smooth function.
We then have
		\begin{equation*}
		a_q\bigl(z',x_n+i\varphi(x'),\oo{z'},x_n-i\varphi(x')\bigr)=0.
		\end{equation*}
		By Malgrange's theorem (see \cite{Ne} and references therein), $\varphi$ is real analytic and thus, since $a_q$ is a polynomial, also real algebraic. Hence, $\partial G$ is Nash algebraic.
	\end{proof}
	
	\textbf{Step 2.}  We now return to the case where $V$ is two-dimensional. We shall prove that $\partial G$ is spherical, where $G$ is as in Theorem \ref{main theorem intro}. Fix $p\in \partial G$, and a canonical extrinsic coordinates chart $(U,z)$ of $V$ at $p$, where $z=(z_1,z_2)$. We again write
	\begin{equation*}
	K_G(z,\bar{z})=K(z,\bar{z})dz\wedge d\oo{z} \quad \mbox{ on } U\cap G,
	\end{equation*}
	where 
$dz=dz_1\wedge dz_2$ and $d\oo{z}=d\oo{z_1}\wedge d\oo{z_2}$. Choose a strongly pseudoconvex domain $D\Subset U\cap G$ such that
	\begin{equation*}
	B(p,\delta)\cap D= B(p,\delta)\cap G \quad \mbox{for some small } \delta>0.
	\end{equation*}
	Here $B(p,\delta)=\{z\in U: \|z-p\|<\delta \}$ is the ball centered at $p$ with radius $\delta$ with respect to the coordinates $(U,z)$. Write $K_D$ for the Bergman kernel of $D$, which is now considered as a function. Then $K_D-K$ extends smoothly across $B(p,\delta)\cap \partial D$ (see \cite{Fe, BoSj}, see also \cite{HuLi} for a nice and detailed proof of this fact). Consequently,
	\begin{equation*}
	K_D(z,\bar{z})=K(z,\bar{z})+h(z,\bar{z})\quad \mbox{on } D,
	\end{equation*}
	where $h(z,\bar{z})$ is real analytic in $D$ and extends smoothly across $B(p,\delta)\cap \partial D$. Let $r$ be a Fefferman defining function of $D$ and express the Fefferman asymptotic expansion of $K_D$ as
	\begin{equation*}
	K_D(z,\bar{z})=\frac{\phi(z,\bar{z})}{r(z)^3}+\psi(z,\bar{z})\log r(z) \quad \mbox{ on } D,
	\end{equation*}
	where $\phi$ and $\psi$ are smooth functions on $D$ that extend smoothly across $B(p,\delta)\cap \partial D$; see \cite{Fe2}.
	Thus,
	\begin{equation}\label{eqnkghr}
	K(z,\bar{z})=\frac{\phi(z,\bar{z})-h(z,\bar{z})r(z)^3}{r(z)^3}+\psi(z,\bar{z})\log r(z) \quad \mbox{ on } D.
	\end{equation}
	As in Step 1, there exist real-valued polynomials $a_1(z,\bar{z}),\cdots, a_q(z,\bar{z})$ in $\mathbb{C}^2\cong\mathbb{R}^4$ with $a_q\neq 0$ for some $q\geq 1$, such that
	\begin{equation*}
	a_qK^q+\cdots+a_1K+a_0=0 \quad \mbox{ on } D.
	\end{equation*}
	If we substitute (\ref{eqnkghr}) into the above equation and  multiply both sides by $r^{3q}$, then
	\begin{equation}\label{vanishing order eq}
	a_q\psi^q r^{3q}(\log r)^q+\sum_{j=0}^{q-1}b_j(\log r)^j=0 \quad \mbox{ on } D,
	\end{equation}
	where all $b_j$ for $0\leq j\leq q-1$ are smooth on $D$ and extend smoothly across $B(p,\delta)\cap \partial D$. We recall the following lemma from \cite{FuWo}.
	\begin{lemma}[\cite{FuWo}]
		Let $f_0(t), \cdots, f_q(t)\in C^{\infty}(-\varepsilon,\varepsilon)$ for $\varepsilon>0$. If
		\begin{equation*}
		f_0(t)+f_1(t)\log t+\cdots+f_q(t)(\log t)^q=0
		\end{equation*}
		for all $t\in (0,\varepsilon)$, then each $f_j(t)$ for $0\leq j\leq q$, vanishes to infinite order at $0$.
	\end{lemma}
	
	It follows from the above lemma and \eqref{vanishing order eq} that the coefficient $\psi$ of of the logarithmic term vanishes to infinite order at $\partial G$ near $p$. Since $G$ is two-dimensional, it follows that $G$ is locally spherical near $p$ by \cite{Gr} (see page 129 where the result is credited to Burns) and \cite{Bou} (see page 23).

\begin{rmk}\label{rmk:Ramadanov}
Recall from the introduction that the sphericity near $p$ above follows from the affirmation of the Ramadanov Conjecture in two dimensions. This is also the only place where the fact that $G$ is two dimensional is essentially used.
\end{rmk}

	
	\textbf{Step 3.} In this step, we will prove there is an algebraic branched covering map $F: \mathbb{B}^2\rightarrow G$ with finitely many sheets. Since we have already shown that $\partial G$ is a Nash algebraic and spherical CR submanifold in $\mathbb{C}^N$, by a theorem of Huang (see Corollary 3.3 in \cite{Hu}), it follows that $\partial G$ is CR equivalent to a CR spherical space form $\partial\mathbb{B}^2/\Gamma$ with $\Gamma\subset U(n)$ a finite group with no fixed points on $\partial \mathbb{B}^2$. In particular, there is a CR covering map $f: \partial\mathbb{B}^2\rightarrow \partial G$ (see the proof of Theorem 3.1 in \cite{Hu} and also \cite{HuJi98}). By Hartogs's extension theorem, $f$ extends as a smooth map $F\colon \overline{\mathbb{B}^2} \to V$, holomorphic in $\mathbb{B}^2$ and sending $\partial \mathbb{B}^2$ onto $\partial G$. The latter implies that $F$ is moreover algebraic by X. Huang's algebraicity theorem \cite{Huang1994}. It is not difficult to see that $F$ sends $\mathbb{B}^2$ into $G$. Since $F$ maps $\partial \mathbb{B}^2$ to $\partial G$, we conclude that $F$ is a proper algebraic mapping $\mathbb{B}^2 \to G$.

\bigskip

{\bf Claim 1.}		$F\colon\mathbb{B}^2\to G$ is surjective.
	\begin{proof}[Proof of Claim 1]
		By the properness of $F$, $F(\mathbb{B}^2)$ is closed in $G$. Let us denote by
		\begin{equation*}
		Z:=\{z\in \mathbb{B}^2: F \mbox{ is not full rank at } z \}.
		\end{equation*}
Since $F$ is a local biholomorphic map at every point of $\partial \mathbb{B}^2$, $Z$ is a finite set. We also note that if $p\in \mathbb{B}^2-Z$, then $F(p)$ is a smooth point of $V$, and $F(p)$ is an interior point of $F(\mathbb{B}^2)$. Assume, in order to reach a contradiction, that $F(\mathbb{B}^2)\neq G$. Since $F(\mathbb{B}^2)$ is closed in $G$, its complement $G\setminus F(\mathbb{B}^2)$ is then a non-empty open subset of $G$. Note that any boundary point of $F(\mathbb{B}^2)$ in $G$ can only be in $F(Z)$. But $F(Z)$ is a finite set, which cannot separate the (non-empty) interior of $F(\mathbb{B}^2)$ and the (non-empty) open complement $G\setminus F(\mathbb{B}^2)$ in the domain $G$. This is the desired contradiction and, hence, $F(\mathbb{B}^2)=G$.
	\end{proof}
	
	Now,  we let $T:=F^{-1}(F(Z)) \supset Z$. Then $T$ is a compact analytic subvariety of $\mathbb{B}^2$ and thus is a finite set. Consider the restriction of $F$:
	\begin{equation*}
	F|_{\mathbb{B}^2-T}: \mathbb{B}^2-T\rightarrow G-F(Z),
	\end{equation*}
	 still denoted by $F$. Clearly, $F$ is a proper surjective map. Since $F$ is also a local biholomorphism, $F$ is a finite covering map.
	
	Note that $\mathbb{B}^2-T$ is simply connected. It follows that the deck transformation group $\widetilde{\Gamma}=\{\widetilde{\gamma}_k \}_{k=1}^m$ of the covering map $F: \mathbb{B}^2-T\rightarrow G-f(Z)$ acts transitively on each fiber. Since each $\widetilde{\gamma}_k$ is a biholomorphism from $\mathbb{B}^2-T$ to $\mathbb{B}^2-T$, it extends to an automorphism of $\mathbb{B}^2$. 
Consequently,
\begin{equation*}
\widetilde{\Gamma}=\bigl\{ \widetilde{\gamma}\in \Aut(\mathbb{B}^2): F\circ \widetilde{\gamma}=F \mbox{ on } \mathbb{B}^2\bigr\}.
\end{equation*}
Recall that $\Gamma$ is the deck transformation group of the original covering map $f:\partial\mathbb{B}^2\rightarrow \partial G$. 
From this, it is clear that we can identify $\Gamma$ with $\widetilde{\Gamma}$. From now on, we will simply use the notation $\Gamma$ for either group.

Note that $Z$ and $T$ are both closed under the action of $\Gamma$, and $(\mathbb{B}^2-T)/\Gamma$ is biholomorphic to $G-f(Z)$.
	
{\bf Claim 2:} If $z, w \in \mathbb{B}^2$ satisfy $F(z)=F(w),$ then $w=\gamma(z)$ for some $\gamma \in \Gamma.$ Consequently, $T=Z$.
	\begin{proof}[Proof of Claim 2]
We only need to prove the first assertion. If both $z, w$ are in $\mathbb{B}^2 -T,$ then the conclusion is clear as $\Gamma$ acts transitively on each fiber of the
covering map $F: \mathbb{B}^2-T\rightarrow G-f(Z)$. Next we assume one of $z$ and $w$ is in $T$. Seeking a contradiction, suppose $w \neq \gamma(z)$ for every $\gamma \in \Gamma.$ Writing $q:=F(z)=F(w),$ there are then points in distinct orbits of $\Gamma$ that are mapped to $q.$ Writing $t$ for the number of  orbits of $\Gamma$ that are mapped to $q$, we must then have $t \geq 2.$ Pick $p_1, \cdots, p_t$ from these $t$ distinct orbits of $\Gamma.$
Since $T$ is a finite set, we can choose, for each $1 \leq i \leq t,$ some disjoint neighborhoods $U_i$ of $p_i$ such that $U_i \cap T  \subseteq \{p_i\}.$ Moreover,
we can make $\gamma(U_i) \cap U_j =\emptyset$ for all $\gamma \in \Gamma$ if $i \neq j.$ Consequently, $F(U_i-\{p_i\}) \cap F(U_j-\{p_j\}) =\emptyset.$ Note there is a small open subset $W$ containing $q$ such that $W \subseteq \cup_{i=1}^t F(\cup_{\gamma \in \Gamma} \gamma(U_i))=\cup_{i=1}^t F(U_i)$. Thus $W-\{q\}\subseteq \cup_{i=1}^t F(U_i- \{p_i\}).$ But the sets $F(U_i-p_i) \cap (W-\{q\})$ are open and disjoint, and we can choose $W-\{q\}$ to be connected. This is a contradiction. Thus we must have $t=1$ and $w=\gamma(z)$ for some $\gamma \in \Gamma.$
	\end{proof}

Recall that $0$ is assumed to be the only fixed point for elements in $\Gamma$. We write $q_0:=F(0)$ and prove that $q_0$ is the only possible singularity in $G$. Also, recall that all singularities of $G$ are assumed to be isolated and normal.
	
{\bf Claim 3.} $G$ can only have a singularity at $q_0$.
	
	\begin{proof}[Proof of Claim 3]
		Suppose $q_1$ is a (normal) singular point in $G$ and $q_1\neq q_0$. Since $F$ is onto, there exists some $p_1\in \mathbb{B}^2$ such that $f(p_1)=q_1$.
		
		First, note that we can find a small neighborhood $U_0$ of $p_1$, and a small neighborhood $W$ of $q_1$ in $G$ such that
$$\text{(i)} ~U_0\cap T=\{p_1\}; \quad \text{(ii)}~F~\text{is injective on}~U_0; \quad \text{(iii)}~W \subseteq F(U_0)~\text{and}~W \cap F(Z)=\{q_1\}.$$


It is easy to see that we can make (i) and (ii) hold. It is guaranteed by Claim 2 (see its proof) that we can find $W \subseteq F(U_0)$; the second condition in (iii) is then easy to satisfy, since $F(Z)$ is a finite set. Now, we let $U:=U_0\cap F^{-1}(W)$, which is an open subset of $\mathbb{B}^2$ containing $p_1$. Then $F: U-\{p_1\}\rightarrow W-\{q_1\}$ is a biholomorphism. We let $g: W-\{q_1\}\rightarrow U-\{p_1\}$ denote its inverse. By the normality of $q_1$, 
we can assume that $g$  is the restriction of some holomorphic map $\widehat{g}$ defined on some open set $\widehat{W}\subset
		\mathbb{C}^N$, where  $\widehat{W}$ contains $W$. Since $g\circ F|_{U-\{p_1\}}$ equals the identity map, $\widehat{g}\circ F$ equals identity on $U$ by continuity. Similarly $F \circ (\widehat{g}|_W)$ equals the identity on $W$.    
Therefore, $q_1$ cannot be a singular point.
	\end{proof}
	
	By Claim 2 and Claim 3, we also see that $T=Z=\{0\} \mbox{ or } \varnothing$. Therefore, $F$ gives a holomorphic algebraic branched covering map from $\mathbb{B}^2$ to $G$ with a possible branch point at $0$. This completes the proof of the "only if" implication in Theorem \ref{main theorem intro}. 	
\end{proof}

\begin{rmk}
We reiterate (see Remark \ref{rmk:Ramadanov} above) that in the above proof, the condition that $\dim V=2$ is only used in the second step where we apply the affirmative solution of the Ramadanov conjecture in $\mathbb{C}^2$ (\cite{Gr}, \cite{Bou}).
\end{rmk}

We shall now prove Corollaries \ref{main corollary intro} and \ref{main corollary 2 intro}.
\begin{proof}[Proof of Corollary \ref{main corollary intro}]
	
By Theorem \ref{main theorem intro}, it follows that $G$ can be realized as a finite ball quotient  $\mathbb{B}^2/\Gamma$ by an algebraic map for some finite unitary group $\Gamma$ with no fixed point on $\partial\mathbb{B}^2$. We must prove that $\Gamma = \{ \id \}$. Suppose not. But, then $G$ must have a singular point (see \cite{Rudin}), which is a contradiction.
\end{proof}

\begin{proof}[Proof of Corollary \ref{main corollary 2 intro}]
The algebraic case follows immediately from Corollary \ref{main corollary intro}. Thus, we only need to consider the rational case.
First, as a consequence of the algebraic case, there exists an algebraic biholomorphic map $f: G\rightarrow \mathbb{B}^2$. It remains to establish that $f$ is in fact rational. This follows immediately from a result by Bell \cite{Bell}. For the convenience of the readers, however, we sketch an independent proof here.
Denote by $K_G$ and $K_{\mathbb{B}^2}$ the Bergman kernels (now considered as functions) of $G$ and $\mathbb{B}^2$, respectively. By the transformation law, they are related as
	\begin{align}\label{K_G transformation eq}
	\begin{split}
	K_G(z,w)
	=&\det\big(Jf(z)\bigr)\cdot K_{B}\bigl(f(z), f(w)\bigr)\cdot\oo{\det\bigl(Jf(w)\bigr)}
	\\
	=&\frac{2!}{\pi^2}\det\big(Jf(z)\bigr)\cdot\oo{\det\bigl(Jf(w)\bigr)}\cdot\frac{1}{\bigl(1-f(z)\cdot \oo{f(w)}\bigr)^3}.
	\end{split}
	\end{align}
	
We may assume $0\in G$ by translating $G$ if necessary and, by composing $f$ with an automorphism of $\mathbb{B}^2$, we may also assume $f(0)=0$. Thus, at $w=0$, we have
	\begin{equation*}
	K_G(z,0)= \frac{2!}{\pi^2}\det\big(Jf(z)\bigr)\cdot\oo{\det\bigl(Jf(0)\bigr)}.
	\end{equation*}
	It follows that
	\begin{equation}\label{det Jf eq}
	  \det\bigl(Jf(z)\bigr)=\det\bigl(Jf(0)\bigr) \frac{K_G(z,0)}{K_G(0,0)}.
	\end{equation}
	In particular, this implies that $K_G(z,0)\neq 0$ for any $z\in D$.
We evaluate \eqref{K_G transformation eq} on the diagonal $w=z$ and use \eqref{det Jf eq} to obtain
	\begin{equation*}
	K_G(z,z)
	=\frac{2!}{\pi^2}\bigl|\det\big(Jf(0)\bigr)\bigr|^2\frac{|K_G(z,0)|^2}{|K_G(0,0)|^2}\frac{1}{\bigl(1-\|f(z)\|^2\bigr)^3}.
	\end{equation*}
Taking the logarithm of both sides yields
	\begin{equation*}
	\log K_G(z,z)+3\log\bigl(1-\|f(z)\|^2\bigr)
	=\log\frac{2!}{\pi^2}+\log\bigl|\det\big(Jf(0)\bigr)\bigr|^2+ \log|K_G(z,0)|^2-\log|K_G(0,0)|^2.
	\end{equation*}
	For $j=1,2$, we apply the derivative $\frac{\partial}{\partial \oo{z_j}}$ to both sides and obtain
	\begin{align*}
	\frac{1}{K_G(z,z)}\frac{\partial K_G(z,z)}{\partial \oo{z_j}}-\frac{3}{1-\|f(z)\|^2}\sum_{i=1}^2\frac{\partial \oo{f_i(z)}}{\partial \oo{z_j}}f_i(z)
	=\frac{1}{\oo{K_G(z,0)}}\frac{\partial \oo{K_G(z,0)}}{\partial \oo{z_j}}.
	\end{align*}
Complexifying the above equation and evaluating it at $w=0$, after rearrangement, we obtain
	\begin{align*}
	\sum_{i=1}^2\frac{\partial \oo{f_i}}{\partial \oo{z_j}}(0)f_i(z)=\frac{1}{3}\left(\frac{1}{K_G(z,0)}\frac{\partial K_G}{\partial \oo{z_j}}(z,0)-\frac{1}{\oo{K_G(0,0)}}\frac{\partial \oo{K_G}}{\partial \oo{z_j}}(0,0)\right).
	\end{align*}
Note this is a linear system for $f(z)=(f_1(z),f_2(z))$ and the coefficient matrix $\oo{Jf(0)}$ is non-singular. By solving this linear system for $f$, it is immediately clear that
the rationality of $K_G$ implies that of $f$.
\end{proof}



\begin{rmk}
Corollary \ref{main corollary intro} implies, in particular, that the Burns--Shnider domains in $\bC^2$ (see page 244 in \cite{BuSh}) cannot have algebraic Bergman kernels. In fact, this holds for any Burns--Shnider domain in $\bC^n$ for $n \geq 2$, which can be seen as follows. By Proposition \ref{boundary algebraic prop}, if the Bergman kernel were algebraic, then the boundary would be Nash algebraic. While this can be seen to not be so by inspection, a contradiction would also be reached by the Huang--Ji Riemann mapping theorem \cite{HuJi02} since the boundary of a Burns--Shnider domain is spherical while the domain itself is not biholomorphic to the unit ball.
\end{rmk}

\section{Counterexample in higher dimension}\label{Sec counterexample}
In this section, we construct a $3-$dimensional reduced Stein space $G$ with only one normal singularity and compact, smooth strongly pseudoconvex boundary,  realized as a relatively compact domain in a complex algebraic variety $V$ in $\mathbb{C}^4$. We will show that its Bergman kernel is algebraic, while $G$ is not biholomorphic to any finite ball quotient $\mathbb{B}^n/\Gamma$, which shows that Theorem \ref{main theorem intro} cannot hold in higher dimensions.

Let $G$ be defined as
\begin{equation*}
	G=\bigl\{w=(w_1,w_2,w_3,w_4)\in\mathbb{C}^4: |w_1|^2+|w_2|^2+|w_3|^2+|w_4|^2<1,\quad w_1w_4=w_2w_3 \bigr\}.
\end{equation*}
Then $G$ is a relatively compact domain in the complex algebraic variety
\begin{equation}\label{variety V}
	V=\bigl\{ w\in \mathbb{C}^4: w_1w_4=w_2w_3\bigr\}.
\end{equation}
Since $G$ is a closed algebraic subvariety of $\mathbb{B}^4\subset \mathbb{C}^4$, $G$ is a reduced Stein space. Note that $0$ is the only singularity of $V$. Moreover it is a normal singularity as it is a hypersurface singularity of codimension 3 ($>2$; see \cite{sha}). It is also easy to verify that $G$ has smooth strongly pseudoconvex boundary in $V$.

\begin{prop}\label{prop:nonspherical}
	The boundary $M=\partial G$ of $G$ is homogeneous and non-spherical.
\end{prop}

\begin{proof}
	Consider the product complex manifold $\mathbb{CP}^1\times \mathbb{CP}^1$. For $j=1,2$, let $\pi_j: \mathbb{CP}^1\times \mathbb{CP}^1\rightarrow \mathbb{CP}^1$ be the projection map to the $j$-th component and let $(L_0,h_0)\rightarrow \mathbb{CP}^1$ be the tautological line bundle $L_0$ with its standard Hermitian metric $h_0$. We set the Hermitian line bundle $(L,h)$ over $\mathbb{CP}^1\times \mathbb{CP}^1$ to be:
	\begin{equation*}
		(L,h):=\pi_1^*(L_0, h_0)\otimes \pi_2^*(L_0, h_0). 
	\end{equation*}
	We begin the proof with the following claim.
	
{\bf Claim 1.} Let $(L,h)\rightarrow \mathbb{CP}^1\times\mathbb{CP}^1$ be as above and let $S(L)\rightarrow \mathbb{CP}^1\times\mathbb{CP}^1$ be its unit circle bundle. Then
		$M$ is CR diffeomorphic to $S(L)$ by the restriction of biholomorphic map.
\begin{proof}[Proof of Claim 1]
	Note that the circle bundle $S(L)\rightarrow\mathbb{CP}^1\times \mathbb{CP}^1$ can be written as
	\begin{equation}\label{circle bundle of CP1 times CP1}
		S(L)=\left\{\Bigl(\lambda(\zeta_1,z_1)\otimes(\zeta_2,z_2), [\zeta_1,z_1],[\zeta_2,z_2]\Bigr): \begin{array}{l}
		[\zeta_1,z_1]\in \mathbb{CP}^1, \quad [\zeta_2,z_2]\in \mathbb{CP}^1\\ |\lambda|^2(|\zeta_1|^2+|z_1|^2)(|\zeta_2|^2+|z_2|^2)=1
		\end{array}   \right\}.
	\end{equation}
Define $F: L\rightarrow \mathbb{C}^4$ as
	\begin{equation}\label{eqnltom}
		F\Bigl(\lambda(\zeta_1,z_1)\otimes(\zeta_2,z_2), [\zeta_1,z_1],[\zeta_2,z_2]\Bigr)=\Bigl( \lambda\zeta_1\zeta_2, \lambda z_1\zeta_2,\lambda\zeta_1z_2,\lambda z_1z_2\Bigr).
	\end{equation}
	Then the map $F$ gives a biholomorphism that sends a neighborhood of $S(L)$ in $L$ to a neighborhood of $M$ in $V\subset \bC^4$. This proves the claim.
\end{proof}

Note that $S(L)$ is homogeneous (see \cite{EnZh}) and non-spherical by Theorem 12 in \cite{Wang}. Thus, $M$ is homogeneous and non-spherical.
\end{proof}

\begin{prop}
	The Bergman kernel form $K_G$ of $G$ is algebraic.
\end{prop}

\begin{proof}
	Set
	\begin{equation}\label{domain Omega}
		\Omega:=\bigl\{ (\lambda, z)=(\lambda,z_1,z_2)\in \mathbb{C}^3: |\lambda|^2(1+|z_1|^2)(1+|z_2|^2)<1 \bigr\}.
	\end{equation}
	Note that $\Omega$ is an unbounded domain with smooth boundary in $\mathbb{C}^3$. Moreover, $\Omega$ has a rational Bergman kernel form $K_{\Omega}$ (see Appendix \ref{Sec Appendix} for a proof of this fact). Define the map $F: \mathbb{C}^3\rightarrow \mathbb{C}^4$ as
	\begin{equation*}
		F(\lambda,z_1,z_2):=(\lambda,\lambda z_1, \lambda z_2, \lambda z_1z_2),
	\end{equation*}
 We note that $F(\mathbb{C}^3)$ is contained in $V$ as defined by \eqref{variety V}. And $F$ is a holomorphic embedding on $\mathbb{C}^3-\{\lambda=0\}$. Moreover, $F(\Omega)\subset G$ and
	\begin{equation*}
		F: \widetilde{\Omega}:=\Omega-\{\lambda=0\}\rightarrow \widetilde{G}:=G-\{w_1=0\}
	\end{equation*}
	is a biholomorphism. By Theorem \ref{Kobayashi thm}, the Bergman kernel form $K_{\widetilde{\Omega}}$ of $\widetilde{\Omega}$ is the restriction (pullback) of $K_{\Omega}$ to $\widetilde{\Omega}$. Thus, $K_{\widetilde{\Omega}}$ is rational. By the transformation law \eqref{BK transformation eq}, we have
	\begin{equation*}
		K_{\widetilde{\Omega}}=(F,F)^*K_{\widetilde{G}}.
	\end{equation*}
This implies that $K_{\widetilde{G}}$ is algebraic (see the equivalent condition (c) in \S 2.1), and thus $K_G$ is also algebraic by Theorem \ref{Kobayashi thm}.
\end{proof}

Before we prove $G$ is not biholomorphic to any finite ball quotient, we pause to study the following bounded domain $U$ in $\mathbb{C}^3$:
\begin{equation*}
	U:=\Bigl\{ (w_1,w_2,w_3)\in \mathbb{C}^3: |w_1|^4+|w_1|^2(|w_2|^2+|w_3|^2)+|w_2w_3|^2<|w_1|^2   \Bigr\}.
\end{equation*}

\begin{prop}\label{propn6}
The domain $U$ has algebraic Bergman kernel and its boundary is non-spherical at every smooth boundary point.
\end{prop}

\begin{proof}
Let $\pi: \mathbb{C}^4\rightarrow \mathbb{C}^3$ be the projection map defined by
\begin{equation*}
	\pi(w_1,w_2,w_3,w_4):=(w_1,w_2,w_3).
\end{equation*}
Let $\oo{G}$ be the closure of $G$ in $\mathbb{C}^4$. Then the image of $\oo{G}$ under the projection $\pi$ is
\begin{align*}
\widehat{U}:=\pi(\oo{G})=&\bigl\{ (w_1,w_2,w_3)\in \mathbb{C}^3: |w_1|^4+|w_1|^2(|w_2|^2+|w_3|^2)+|w_2w_3|^2 \leq |w_1|^2, w_1\neq 0 \bigr\}
	\\
	&\quad\cup  \bigl\{ (0,w_2,w_3)\in \mathbb{C}^3: |w_2|^2+|w_3|^2\leq 1, w_2w_3=0  \bigr\}
	\\
	=&\bigl\{ (w_1,w_2,w_3)\in \mathbb{C}^3: |w_1|^4+|w_1|^2(|w_2|^2+|w_3|^2)+|w_2w_3|^2 \leq |w_1|^2, |w_2|^2+|w_3|^2\leq 1 \bigr\}.
\end{align*}

Note that $\widehat{U}^{\mathrm{o}}=U \quad \mbox{and}\quad \widehat{U}=\oo{U}$,
where $\widehat{U}^{\mathrm{o}}$ denotes the interior of $\widehat{U}$ . 
But $U \neq \pi(G)$. On the other hand if we remove the variety $\{w_1=0\}$, then the projection map
\begin{equation*}
	\pi: G-\{w_1=0\}\rightarrow U
\end{equation*}
is an algebraic biholomorphism. Consequently, by Theorem \ref{Kobayashi thm} the Bergman kernel form $K_U$ of $U$ is algebraic. This proves the first part of the proposition.

To prove the second part of the proposition (i.e., the non-sphericity), we observe that the boundary $\partial U$ of $U$ is given by
\begin{align*}
	\partial U
	=&\bigl\{ (w_1,w_2,w_3)\in \mathbb{C}^3: |w_1|^4+|w_1|^2(|w_2|^2+|w_3|^2)+|w_2w_3|^2 = |w_1|^2, w_1\neq 0 \bigr\}
	\\
	&\quad\cup  \bigl\{ (0,w_2,w_3)\in \mathbb{C}^3: |w_2|^2+|w_3|^2\leq 1, w_2w_3=0  \bigr\}
	\\
	=&\bigl\{ (w_1,w_2,w_3)\in \mathbb{C}^3: |w_1|^4+|w_1|^2(|w_2|^2+|w_3|^2)+|w_2w_3|^2 = |w_1|^2, |w_2|^2+|w_3|^2\leq 1 \bigr\}.
\end{align*}
Write
\begin{equation*}
	\partial U=\bigl( \partial U \cap \{w_1\neq 0 \} \bigr)\cup \bigl( \partial U \cap \{w_1 = 0 \} \bigr).
\end{equation*}
Since the projection map $\pi$ is a biholomorphism from $\oo{G}-\{w_1=0\}$ to $\widehat{U}-\{w_1=0\}$, every point $p\in \partial U \cap \{w_1\neq 0\}$ is a smooth point of $\partial U$, and, moreover, $\partial U$ is strictly pseudoconvex and non-spherical at $p$. We note that a defining function for $\partial U$ near $p$ is given by
\begin{equation*}
	\rho=|w_1|^4+|w_1|^2(|w_2|^2+|w_3|^2)+|w_2w_3|^2-|w_1|^2.
\end{equation*}
Furthermore, it is easy to verify that every other point $q \in \partial U \cap \{w_1 = 0 \}$ is not a smooth boundary point of $U$. This proves the second part of the assertion.
\end{proof}


We are now ready to show that $G$ is indeed a counterexample to the conclusion of Theorem \ref{main theorem intro} in three dimensions.

\begin{prop}
	$G$ is not biholomorphic to any finite ball quotient.
\end{prop}
\begin{proof}
Seeking a contradiction, we suppose $G$ is biholomorphic to a finite ball quotient $\mathbb{B}^3/\Gamma$, where $\Gamma \subset U(n)$ is a finite unitary group. We realize $\mathbb{B}^3/\Gamma$ as the image $G_0\subset \mathbb{C}^N$ of $\mathbb{B}^3$ under the basic map $Q$ associated to $\Gamma$, where $Q=(p_1,\cdots, p_N): \mathbb{C}^3\rightarrow \mathbb{C}^N$ gives a proper map from $\mathbb{B}^3$ to $G_0.$
Let $F$ be a biholomorphism from $G_0\cong \mathbb{B}^3/\Gamma$ to $G$. Then there is an analytic variety $W_0\subset G_0$ such that
	\begin{equation*}
		F: G_0-W_0\rightarrow G-\{w_1=0\} \mbox{ is a biholomorphism}.
	\end{equation*}
There also exists an analytic variety $W$ such that $W=Q^{-1}(W_0)$ and thus $Q: \mathbb{B}^3-W\rightarrow G_0-W_0$ is proper and onto.
Set
	\begin{equation*}
		f:=\pi\circ F\circ Q : \mathbb{B}^3-W \rightarrow U=\pi(G-\{w_1=0\}),
	\end{equation*}
where $\pi$ is the projection defined in the proof of Proposition \ref{propn6} and is a biholomorphism from $G-\{w_1=0\}$ to $U$. Note that $f$ is proper. Since $U\subset \mathbb{C}^3$, we can write $f$ as $(f_1,f_2,f_3)$.
	
{\bf Claim 2.} There is a sequence $\{\zeta_i\}\subset \mathbb{B}^3-W$ with $\zeta_i\rightarrow \zeta^*\in \partial \mathbb{B}^3-\oo{W}$ such that
		\begin{equation*}
			f(\zeta_i) \rightarrow p^*\in \partial U \cap \{w_1\neq 0\}.
		\end{equation*}

\begin{proof}[Proof of Claim 2]
	Suppose not. Then for any $\{\zeta_i\}\subset \mathbb{B}^n-W$ with $\zeta_i\rightarrow \zeta^* \in \partial \mathbb{B}^n-\oo{W}$, every convergent subsequence of $f(\zeta_i)$ converges to some point in $\partial U\cap \{w_1=0\}$. That is to say, if $f(\zeta_{i_k})$ is convergent, then $f_1(\zeta_{i_k})\rightarrow 0$. Note that $U$ is bounded. Thus, $f_1(\zeta_i)\rightarrow 0$ for any $\{\zeta_i\}\subset \mathbb{B}^3-W$ with $\zeta_i\rightarrow\zeta^*\in \partial\mathbb{B}^3-\oo{W}$. By a standard argument using analytic disks attached to $\partial\mathbb{B}^3-\oo{W}$, we see that
	\begin{equation*}
		f_1=0 \quad \mbox{ on } \mathbb{B}^3-W.
	\end{equation*}
	This is a contradiction.
\end{proof}

Let $\zeta_i, \zeta^*$	and $p^*$ be as in Claim 2. Note that $\zeta^*$ is a smooth strictly pseudoconvex boundary point of $\mathbb{B}^3-W$, and $p^*$ is a smooth strictly pseudoconvex boundary point of $U$ (see the proof of Proposition \ref{propn6}).
It follows from \cite{FR} (see page 239) that $f$ extends to a H\"older-$\frac{1}{2}$ continuous map on a neighborhood of $\zeta^*$ in $\oo{\mathbb{B}^3}$.  Since $f$ is proper $\mathbb{B}^3-W\to U$, its extension to the boundary is a (H\"older-$\frac{1}{2}$) continuous, nonconstant CR map sending a piece of $\partial \mathbb{B}^3$ containing $\zeta^*$ to a piece of $\partial U$ containing $p^*$.
By \cite{PT}, $f$ extends holomorphically to a neighborhood of $\zeta^*$, since both boundaries are real analytic (in fact, real-algebraic). Now, since $f$ is non-constant and sends a strongly pseudoconvex hypersurface to another, it must be a CR diffeomorphism, which would mean that $\partial U$ is locally spherical near $p^*$. This contradicts Proposition \ref{prop:nonspherical}.	
\end{proof}

We conclude this section by a couple of remarks.

\begin{rmk}
Since the Bergman kernel forms of $G$ and $U$ are algebraic, it follows from the proof of Theorem \ref{main theorem intro} in Section 5 (see Step 2) that the coefficients of the logarithmic term in Fefferman's expansions of $K_G$ and $K_V$ both vanish to infinite order at every smooth boundary point. The reduced normal Stein space $G$ gives the counterexample mentioned in Remark \ref{rmk counterexample intro}. The domain $U\subset \bC^3$ establishes the following fact, which implies that the Ramadanov conjecture fails for non-smooth domains in higher dimension.

{\it{There exists a bounded domain in $\mathbb{C}^3$ with smooth, real-algebraic boundary away from a 1-dimensional {complex} curve such that every smooth boundary point is strongly pseudoconvex and non-spherical, while the coefficient of the logarithmic term in Fefferman's asymptotic expansion of the Bergman kernel vanishes to infinite order at every smooth boundary point.}}
\end{rmk}

\begin{rmk}
	Using the same idea as in the above example, we can actually construct significantly more general examples of higher dimensional domains in affine algebraic varieties $V\subset \mathbb{C}^N$ with similar properties. Indeed, let $X$ be a compact Hermitian symmetric space of rank at least $2$. Write
	\begin{equation*}
		X=X_1\times \cdots \times X_t, \quad t\geq 1,
	\end{equation*}
	where $X_1, \cdots, X_t$ are the irreducible factors of $X$. Fix a \k-Einstein metric $\omega_j$ on $X_j$ and let $(\widehat{L}_j, \widehat{h}_j)$ be the top exterior product $\Lambda^nT^{1,0}$ of the holomorphic tangent bundle over $X_j$ with the metric induced from $\omega_j$. Then there is a homogeneous line bundle $(L_j,h_j)$ with a Hermitian metric $h_j$ such that its $p_j$-th tensor power gives $(\widehat{L}_j, \widehat{h}_j)$, where $p_j$ is the genus of  $X_j.$ (see \cite{EnZh} for more details).
	
	Let $\pi_j$ be the projection from $X$ onto the $j$-th factor $X_j$ for $1\leq j\leq t$. Define the line bundle $L$ over $X$ with a Hermitian metric $h$ to be:
	\begin{equation*}
		(L,h):=\pi_1^*(L_1, h_1)\otimes \cdots \otimes \pi_t^*(L_t, h_t).
	\end{equation*}
	Let $(L^*, h^*)$ be the dual line bundle of $(L,h)$. Write $D(L^*)$ and $S(L^*)$ for the associated unit disc and unit circle bundle. The specific example above is the special case $t=2$ and $X_1=X_2=\mathbb{CP}^1$. Proceeding as in that example, one finds that there is a canonical way to map $L^*$ to $\mathbb{C}^N$, for some $N$, induced by the minimal embedding of $X$ into some complex projective space (see \cite{FHX}). If we denote this map $L^*\to \mathbb{C}^N$ by $F$ (in the example above, the map $F$ is as given by \eqref{eqnltom}), then $F$ sends the zero section of $L^*$ to the point $0$ and is a holomorphic embedding away from the zero section. It follows that the image of $D(L^*)$ under the map $F$ is a domain $G$ with a singular point at $0$. The boundary of $G$ is given by the image of $S(L^*)$. It is not spherical since $S(L^*)$ is not by \cite{Wang}. Moreover, as the Bergman kernel form of $D(L^*)$ is algebraic by \cite{EnZh}, the Bergman kernel form of $G$ is also algebraic by Theorem \ref{Kobayashi thm}.	
\end{rmk}

\section{Appendix}\label{Sec Appendix}
In this section, we will prove the claim that the Bergman kernel of the domain $\Omega$ in $\mathbb{C}^3$ as defined in \eqref{domain Omega} is rational.
This fact actually follows from a general theorem in \cite{EnZh} (see Theorem 3.3 in \cite{EnZh} and its proof). We include here a proof in this particular example for the convenience of readers and self-containedness of this paper. In fact, we shall compute the Bergman kernel of $\Omega$ explicitly (Theorem \ref{thm:OmegaBergman} below).

Recall that
\begin{equation}\label{eq:Omega}
		\Omega:=\bigl\{ (z,\lambda)=(z_1,z_2,\lambda)\in \mathbb{C}^3: |\lambda|^2(1+|z_1|^2)(1+|z_2|^2)<1 \bigr\}.
\end{equation}
We let
\begin{equation*}
	h(z):=(1+|z_1|^2)(1+|z_2|^2),
\end{equation*}
and denote the defining function by
\begin{equation*}
	\rho(z,\lambda):=|\lambda|^2(1+|z_1|^2)(1+|z_2|^2)-1.
\end{equation*}

We recall that the Bergman space on $\Omega$ is defined as
\begin{equation}
A^2(\Omega):=\bigl\{f(z,\lambda) \mbox{ is holomorphic in } \Omega: i\int_{\Omega} |f(z,\lambda)|^2 dz\wedge d\lambda\wedge d\oo{z} \wedge d\oo{\lambda}<\infty \bigr\},
\end{equation}
and let
\begin{equation}
A^2_m(\Omega):=\bigl\{ f(z) \text{ {\rm is holomrphic in $\bC^2$}}\colon \lambda^m f(z)\in A^2(\Omega)  \bigr\}.
\end{equation}
Note that the $L^2$ norm of $\lambda^mf(z)$ is given by
\begin{align*}
\|\lambda^m f(z)\|^2=&i\int_{\Omega} |\lambda|^{2m}|f(z)|^2 dz\wedge d\lambda \wedge d\oo{z}\wedge d\oo{\lambda}
\\
=&\int_{z\in \mathbb{C}^2}\Bigl(\int_{|\lambda|^2<h(z)^{-1}} |\lambda|^{2m} i\,d\lambda \wedge d\oo{\lambda} \Bigr) |f(z)|^2dz\wedge d\oo{z}
\\
=&\int_{z\in \mathbb{C}^2}\Bigl(\int_{|\lambda|^2<h(z)^{-1}} |\lambda|^{2m} i\,d\lambda \wedge d\oo{\lambda} \Bigr) |f(z)|^2dz\wedge d\oo{z}.
\end{align*}
We can rewrite the inner integral as follows:
\begin{align*}
\int_{|\lambda|^2<h(z)^{-1}} |\lambda|^{2m} i\,d\lambda \wedge d\oo{\lambda}
=\int_0^{2\pi}\int_0^{\frac{1}{\sqrt{h(z)}}} r^{2m}2r drd\theta
=2\pi\int_0^{\frac{1}{h(z)}} r^{m} dr
=\frac{2\pi}{m+1} h(z)^{-(m+1)}.
\end{align*}
Thus,
\begin{equation}\label{norm}
\|\lambda^m f(z)\|^2=\frac{2\pi}{m+1} \int_{\mathbb{C}^2} |f(z)|^2 h(z)^{-(m+1)} dz\wedge d\oo{z}.
\end{equation}
If we introduce the weighted Bergman space
\begin{equation*}
A^2(\mathbb{C}^2, h^{-(m+1)})=\bigl\{ f(z) \mbox{ is holomorphic in }\mathbb{C}^2: \int_{\mathbb{C}^2} |f(z)|^2 h(z)^{-(m+1)} dz\wedge d\oo{z}<\infty \bigr\},
\end{equation*}
then
\begin{equation}\label{weighted L2 space}
A^2_m(\Omega)=\bigl\{\lambda^m f(z): f(z) \in A^2(\mathbb{C}^2, h^{-(m+1)}) \bigr\}.
\end{equation}
We note that
$A_{m_1}^2(\Omega)$ and $A_{m_2}^2(\Omega)$ are orthogonal to each other if $m_1\neq m_2$. We can therefore orthogonally decompose $A^2(\Omega)$ into a direct sum as follows.
\begin{lemma}\label{orthogonal decomposition}
	\begin{equation*}
	A^2(\Omega)=\bigoplus_{m=0}^{\infty} A_m^2(\Omega).
	\end{equation*}
\end{lemma}

\begin{proof}
	Let $f(z,\lambda)\in A^2(\Omega)$. If we fix $z\in \mathbb{C}^2$, then $\lambda$ is contained in the disc $\{\lambda\in \mathbb{C}, |\lambda|^2<h(z)^{-1}\}$. By taking the Taylor expansion at $\lambda=0$, we obtain
	\begin{equation*}
	f(z,\lambda)=\sum_{j=0}^{\infty} a_j(z)\lambda^j, \quad \mbox{ for } |\lambda|^2< h(z)^{-1},
	\end{equation*}
	where each $a_j(z)$ is holomorphic on $\mathbb{C}^2$. We shall first write $\|f(z,\lambda)\|^2$ in terms of $\{a_j(z)\}_{j=0}^{\infty}$. We have
	\begin{align*}
	\|f(z,\lambda)\|^2=&\int_{z\in \mathbb{C}^2}\left(\int_{|\lambda|^2<h^{-1}(z)} |f(z,\lambda)|^2 i\,d\lambda\wedge d\oo{\lambda} \right) dz\wedge d\oo{z}.
	\end{align*}
	The inner integral can be computed as
	\begin{align*}
	\int_{|\lambda|^2<h^{-1}(z)} |f(z,\lambda)|^2 i\,d\lambda \wedge d\oo{\lambda}
	=&\lim_{\eps\rightarrow 0^+}\sum_{s=0}^{\infty}\sum_{t=0}^{\infty}a_s(z)\oo{a_t(z)}\int_{|\lambda|^2<h^{-1}(z)-\eps} \lambda^s\oo{\lambda^t} i\,d\lambda \wedge d\oo{\lambda}
	\\
	=&\lim_{\eps\rightarrow 0^+}\sum_{j=0}^{\infty}\frac{2\pi }{j+1}|a_j(z)|^2\bigl(h(z)^{-1}-\eps\bigr)^{j+1}
	\\
	=&\sum_{j=0}^{\infty}\frac{2\pi }{j+1}|a_j(z)|^2 h(z)^{-(j+1)},
	\end{align*}
	where the last equality follows from the Monotone Convergence theorem. Therefore,
	\begin{align*}
	\|f(z,\lambda)\|^2=&\sum_{j=0}^{\infty}\frac{2\pi }{j+1}\int_{z\in \mathbb{C}^2}|a_j(z)|^2 h(z)^{-(j+1)} dz\wedge d\oo{z},
	\end{align*}
	which immediately implies that each $a_j(z)$ is contained in $A^2(\mathbb{C}^2, h^{-(j+1)})$.
	
	Suppose $f(z,\lambda)\perp A^2_m(\Omega)$. Then for any $\lambda^m g(z)\in A^2_m(\Omega)$,
	\begin{align*}
	0=&i\int_{\Omega} f(z,\lambda)\oo{\lambda^m g(z)} dz\wedge d\lambda \wedge d\oo{z}\wedge d\oo{\lambda}
	\\
	=&\int_{z\in \mathbb{C}^2}\left(\int_{|\lambda|^2<h(z)^{-1}} f(z,\lambda)\oo{\lambda^m} \,i\,d\lambda \wedge d\oo{\lambda} \right) \oo{g(z)}dz\wedge d\oo{z}.
	\end{align*}
	The inner integral can be computed as follows
	\begin{align*}
	\int_{|\lambda|^2<h^{-1}(z)} f(z,\lambda)\oo{\lambda^m} i\,d\lambda \wedge d\oo{\lambda}
	=&\lim_{\eps\rightarrow 0^+}\sum_{j=0}^{\infty}a_j(z)\int_{|\lambda|^2<h(z)^{-1}-\eps} \lambda^j\oo{\lambda^m} i\,d\lambda \wedge d\oo{\lambda}
	\\
	=&\lim_{\eps\rightarrow 0^+}\frac{2\pi }{m+1}a_m(z)\bigl(h(z)^{-1}-\eps\bigr)^{m+1}
	\\
	=&\frac{2\pi}{m+1} a_m(z) h(z)^{-(m+1)}.
	\end{align*}
	Therefore,
	\begin{align*}
	0=&\frac{2\pi}{m+1}\int_{z\in \mathbb{C}^2} a_m(z)  \oo{g(z)}h(z)^{-(m+1)}dz\wedge d\oo{z}, \quad \mbox{ for any } g\in A^2(\mathbb{C}^2, h^{-(m+1)}).
	\end{align*}
	Since $a_m$ belongs to the space $A^2(\mathbb{C}^2, h^{-(m+1)})$, we get $a_m=0$. Therefore, the direct sum of $A_m^2(\Omega)$ for $0\leq m<\infty$ generates $A^2(\Omega)$.
\end{proof}

Since $A^2_m(\Omega)$ can be identified with the weighted Bergman space $A^2(\mathbb{C}^2, h^{-(m+1)})$ as in \eqref{weighted L2 space}, we can find an explicit orthonormal basis and compute its reproducing kernel.

\begin{prop}\label{K_m^* CPn}
	Let $m \geq 1.$ The reproducing kernel of $A^2_m(\Omega)$ is
	\begin{equation}
	K_m^*(z,\lambda,\oo{w},\oo{\tau})
	=\frac{(m+1)m^2}{(2\pi)^3}\lambda^m\oo{\tau}^m(1+z_1\oo{w_1})^{m-1}(1+z_2\oo{w_2})^{m-1},
	\end{equation}
	where $(z,\lambda), (w,\tau)$ are points in $\Omega$.
\end{prop}
\begin{proof}
	Denote
	\begin{equation*}
	z^{\alpha}=z_1^{\alpha_1}z_2^{\alpha_2}, \quad \mbox{ for any multi-index } \alpha=(\alpha_1,\alpha_2)\in \mathbb{Z}_{\geq 0}^2.
	\end{equation*}
	By \eqref{weighted L2 space}, since $\Omega$ is Reinhardt, it is easy to see that
	\begin{equation*}
	\bigl\{ \lambda^m z^{\alpha}: z^{\alpha} \in A^2(\mathbb{C}^2, h^{-(m+1)}) \bigr\}
	\end{equation*}
	forms an orthogonal basis of $A^2_m(\Omega)$. We shall compute the norm for each $\lambda^m z^{\alpha}$. Using \eqref{norm}, we have
	\begin{align*}
	\|\lambda^m z^{\alpha}\|^2
	=&\frac{2\pi}{m+1} \int_{\mathbb{C}^2} |z^{\alpha}|^2 (1+|z_1|^2)^{-(m+1)}(1+|z_2|^2)^{-(m+1)} dz\wedge d\oo{z}
	\\
	=&\frac{2\pi}{m+1} \int_{\mathbb{C}} |z_1|^{2\alpha_1} (1+|z_1|^2)^{-(m+1)}\,i\, dz_1\wedge d\oo{z_1} \cdot\int_{\mathbb{C}} |z_2|^{2\alpha_2} (1+|z_2|^2)^{-(m+1)}\,i\, dz_2\wedge d\oo{z_2}
	\\
	=&\frac{(2\pi)^3}{m+1} \int_0^{\infty} r_1^{\alpha_1} (1+r_1)^{-(m+1)} dr_1 \cdot\int_0^{\infty} r_2^{\alpha_2} (1+r_2)^{-(m+1)} dr_2.
	\end{align*}
	By the elementary integral identity
	\begin{equation}\label{element integral}
	\int_0^{\infty} r^p \frac{1}{(1+r)^q}dr=\frac{(q-p-2)!p!}{(q-1)!}, \quad \mbox{ for any nonnegative integers $p, q$ with } q\geq p+2,
	\end{equation}
	we get
	\begin{equation*}
		\|\lambda^m z^{\alpha}\|^2=\begin{dcases}
		\frac{(2\pi)^3}{m+1}\frac{(m-\alpha_1-1)!(m-\alpha_2-1)!\alpha!}{m!^2} & \mbox{if }\alpha_1, \alpha_2\leq m-1,
		\\
		+\infty & \mbox{otherwise}.
		\end{dcases}
	\end{equation*}
Thus,  $\bigl\{ \frac{\lambda^m z^{\alpha}}{\|\lambda^mz^{\alpha}\|}: 0\leq \alpha_1, \alpha_2\leq m-1 \bigr\}$ is an orthonormal basis of $A^2_m(\Omega)$,  and the reproducing kernel of $A^2_m(\Omega)$ is given by
	\begin{align*}
	K_m^*(z,\lambda,\oo{w},\oo{\tau})
	&=\sum_{0\leq \alpha_1, \alpha_2\leq m-1} \frac{z^{\alpha}\lambda^m\oo{w}^{\alpha}\oo{\tau}^m}{\|z^{\alpha}\lambda^m\|^2}\\
	&=\frac{(m+1)m^2}{(2\pi)^3}\lambda^m\oo{\tau}^m\sum_{\alpha_1=0}^{m-1}\binom{m-1}{\alpha_1}z_1^{\alpha_1}\oo{w_1}^{\alpha_1} \sum_{\alpha_2=0}^{m-1}\binom{m-1}{\alpha_2}z_2^{\alpha_2}\oo{w_2}^{\alpha_2}
	\\
	&=\frac{(m+1)m^2}{(2\pi)^3}\lambda^m\oo{\tau}^m(1+z_1\oo{w_1})^{m-1}(1+z_2\oo{w_2})^{m-1}.
	\end{align*}
\end{proof}

Now we are ready to compute the Bergman kernel form of $\Omega$.
\begin{thm}\label{thm:OmegaBergman}
	The Bergman kernel form of the domain $\Omega\subset\bC^3$ in $\eqref{eq:Omega}$ is given by
	\begin{equation*}
	K_\Omega(z,\lambda,\oo{w},\oo{\tau})=i K^*(z,\lambda,\oo{w},\oo{\tau}) dz\wedge d\lambda\wedge d\oo{w}\wedge d\oo{\tau},
	\end{equation*}
	where
	\begin{equation*}
	K^*(z,\lambda,\oo{w},\oo{\tau})=\sum_{m=1}^{\infty}\frac{(m+1)m^2}{(2\pi)^3}\lambda^m\oo{\tau}^m(1+z_1\oo{w_1})^{m-1}(1+z_2\oo{w_2})^{m-1}.
	\end{equation*}
	It can be written in terms of the complexified defining function $$
\rho(z,\lambda,\oo{w},\oo{\tau})=\lambda\oo{\tau}(1+z_1\oo{w_1})(1+z_2\oo{w_2})-1
$$
as
	\begin{equation*}
	K^*(z,\lambda,\oo{w},\oo{\tau})=\frac{1}{(2\pi)^3}\Bigl(\frac{4\lambda\oo{\tau}}{\rho(z,\lambda,\oo{w},\oo{\tau})^{3}}+\frac{6\lambda\oo{\tau}}{\rho(z,\lambda,\oo{w},\oo{\tau})^{4}}\Bigr).
	\end{equation*}
\end{thm}
\begin{proof}
	By Lemma \ref{orthogonal decomposition}, we immediately get the reproducing kernel of $A^2(\Omega)$ by adding up the reproducing kernels of $A^2_m(\Omega)$ for all $m$. Since $A^2_0(\Omega)=\{0\}$, we obtain
	\begin{align*}
	K^*(z,\lambda,\oo{w},\oo{\tau})
	=&\sum_{m=1}^{\infty}K_m^*(z,\lambda,\oo{w},\oo{\tau})\\
	=&\sum_{m=0}^{\infty}\frac{(m+2)(m+1)^2}{(2\pi)^3}\lambda^{m+1}\oo{\tau}^{m+1}(1+z_1\oo{w_1})^{m}(1+z_2\oo{w_2})^{m}.
	\end{align*}
It remains to write $K^*(z,\lambda,\oo{w},\oo{\tau})$ in terms of the defining function $\rho(z,\lambda,\oo{w},\oo{\tau})$. We use the Taylor expansion of $1/(1-x)^{j+1}$ for $0\leq j\leq 3$ to obtain
	\begin{equation*}
	\frac{1}{(-\rho(z,\lambda,\oo{w},\oo{\tau}))^{j+1}}=\frac{1}{(1-(1+z_1 \oo{w_1})(1+z_2 \oo{w_2})\lambda\oo{\tau})^{j+1}}=\sum_{m=0}^{\infty}\binom{m+j}{j}(1+z_1 \oo{w_1})^m(1+z_2 \oo{w_2})^m\lambda^m\oo{\tau}^m.
	\end{equation*}
	Note that $(m+2)(m+1)^2$ is a polynomial in $m$ of degree $3$. Since $\{\binom{m+j}{j}\}_{j=0}^{3}$ is a basis of polynomials in $m$ with degree $\leq 3$, we can write
	\begin{equation*}
	(m+2)(m+1)^2=\sum_{j=0}^{3}a_j\binom{m+j}{j}.
	\end{equation*}
One can check that the coefficients are given by $a_0=a_1=0$, $a_2=-4$ and $a_3=6$. Therefore,
	\begin{align*}
	K^*(z,\lambda,\oo{w},\oo{\tau})
	=&\frac{1}{(2\pi)^3}\sum_{m=0}^{\infty}\sum_{j=0}^{3}a_j\binom{m+j}{j}\lambda^{m+1}\oo{\tau}^{m+1}(1+z_1\oo{w_1})^{m}(1+z_2\oo{w_2})^{m}
	\\
	=&\frac{1}{(2\pi)^3}\sum_{j=0}^{3}a_j\frac{\lambda\oo{\tau}}{(-\rho(z,\lambda,\oo{w},\oo{\tau}))^{j+1}},
	\end{align*}
and the result follows.
\end{proof}

\bibliographystyle{plain}
\bibliography{references}

\end{document}